\documentclass{amsart}
\usepackage{amsmath}
\usepackage{amssymb}
\usepackage{amscd}
\usepackage{graphicx}
\usepackage{latexsym}
\usepackage{psfrag}
\usepackage{framed, color}
\usepackage{comment}

\newtheorem{thm}{Theorem}[section]

\newtheorem{cor}[thm]{Corollary}
\newtheorem{prop}[thm]{Proposition}

\theoremstyle{definition}

\theoremstyle{remark}
\newtheorem{rk}[thm]{Remark}

\newtheorem{example}[thm]{Example}

\makeatletter
 
 \@addtoreset{equation}{section}
\makeatother

\newcommand{\CPb}{\overline{\mathbb{CP}}{}^{2}}
\newcommand{\CP}{{\mathbb{CP}}{}^{2}}

\newcommand{\Int}{\mathop{\mathrm{Int}}\nolimits}
\newcommand{\rank}{\mathop{\mathrm{rank}}\nolimits}

\newcommand{\R}{\mathbb{R}}
\newcommand{\Z}{\mathbb{Z}}

\newcommand{\bang}{\textcolor{red}{\textbf{\Large !\,}}}
\newcommand{\id}{\mathop{\mathrm{id}}\nolimits}

\def \x {\times}

\begin{document}

\title[Simplifying indefinite fibrations on $4$--manifolds] {Simplifying indefinite fibrations on $4$--manifolds}
\vspace{0.2in} 

\author[R. \.{I}. Baykur]{R. \.{I}nan\c{c} Baykur}
\address{Department of Mathematics and Statistics, University of Massachusetts, Amherst, MA 01003-9305, USA}
\email{baykur@math.umass.edu}

\author[O. Saeki]{Osamu Saeki}
\address{Institute of Mathematics for Industry,
Kyushu University, Motooka 744, Nishi-ku, Fukuoka 819-0395,
Japan}
\email{saeki@imi.kyushu-u.ac.jp }

\begin{abstract}
We present explicit algorithms for simplifying the topology of indefinite fibrations on $4$--manifolds, which include broken Lefschetz fibrations and indefinite Morse $2$--functions. The algorithms consist of sequences of moves, which modify indefinite fibrations in smooth $1$--parameter families. In particular, given an arbitrary broken Lefschetz fibration, we show how to turn it to one with directed and embedded round (indefinite fold) image, and to one with all the fibers and the round locus connected. We also show how to realize any given
null-homologous $1$--dimensional submanifold with prescribed local models for its components as the round locus of such a broken Lefschetz fibration. These algorithms allow us to give purely topological and constructive proofs of the existence of simplified broken Lefschetz fibrations and Morse $2$--functions on general $4$--manifolds, and a theorem of Auroux--Donaldson--Katzarkov on the existence of broken Lefschetz pencils with directed embedded round image on near-symplectic $4$--manifolds. We moreover establish a correspondence between broken Lefschetz fibrations and \mbox{Gay--Kirby} trisections of $4$--manifolds, and show the existence of simplified trisections on all $4$--manifolds. Building on this correspondence, we provide several new constructions of trisections, including infinite families of genus--$3$ trisections with homotopy inequivalent total spaces, and exotic same genera trisections of \mbox{$4$--manifolds} in the homeomorphism classes of complex rational surfaces.
\end{abstract}

\maketitle

\tableofcontents

\section{Introduction}

An \emph{indefinite fibration} is a smooth map from a compact smooth oriented \linebreak $4$--manifold onto an orientable surface, with simplest types of stable and unstable singularities; its singular locus consists of indefinite folds and cusps along embedded circles (\emph{round locus}), and Lefschetz singularities along a disjoint discrete set. Its fibers are smooth orientable surfaces and singular surfaces obtained by collapsing embedded loops on them. \mbox{The class} of indefinite fibrations\footnote{The name ``indefinite fibrations'' is indeed a mutual compromise for ``indefinite generic maps'' (indefinite Morse $2$--functions) and ``broken Lefschetz fibrations''. Other authors used ``wrinkled fibrations'' \cite{L} and ``broken fibrations'' \cite{W} for the same class of maps, among others.} naturally includes \emph{broken Lefschetz fibrations} (when there are no cusps) and \emph{indefinite Morse $2$-- functions} (no Lefschetz singularities), both of which received a hefty amount of attention over the past decade; e.g.\ \cite{AK, ADK, B2, B1, B3, B4, BH, BK, SB, BeH, GK1, GK3, GK2, H1, H2, H3, HS1, HS2, Hughes, KMT, L, P, P1, Sa, Sa2, W, W2}.  Moreover, a special class of generic maps, which are ``almost'' indefinite, except for a single definite fold circle, underline the recently emerging theory of \emph{trisections} of $4$--manifolds; e.g. \cite{AGK, Castro, CGP, Gay, GK4, MSZ, MZ, MZ2, RS}.

In this paper, we provide explicit algorithms which immensely \mbox{\emph{simplify}} the topology of indefinite fibrations, and in turn, we present purely topological and constructive proofs of the existence  of simplified broken Lefschetz fibrations, pencils, and trisections.

Our algorithms consist of sequences of moves that modify indefinite fibrations 
in smooth $1$--parameter families. We describe several procedures of this type,
especially for maps onto the $2$--sphere,
to derive an indefinite fibration ---roughly--- with the property that \emph{every fiber}
\vspace{-0.05cm}
\begin{itemize}
\item can be obtained from a fixed regular fiber by a sequence of fiberwise 
\linebreak \mbox{$2$--handle} attachments \,(\emph{directed}),
\item contains at most one fold or cusp point (\emph{embedded} round image),
\item is connected  \,(\emph{fiber-connected}),
\item if regular, is either of genus $g$ or $g-1$ \,(\emph{simplified}),
\end{itemize}
\vspace{-0.05cm}
and we achieve these properties cumulatively.  Here, what makes an indefinite fibration ``simpler'' is in essence quantified by the complexity of the handle decomposition induced by the fibration; e.g.\ the second property allows one to describe the $4$--manifold by \emph{round $2$--handle} (circle times a $2$--handle) attachments to a simple handlebody for a Lefschetz fibration over the $2$--disk, and the fourth makes it possible to do it with a single round $2$--handle \cite{B1}; cf.\ \cite{GK3}. Simplified fibrations induce decompositions of the $4$--manifold into elementary fibered cobordisms between surface bundles over circles,  suitable for calculating Floer theoretic invariants, such as  Heegaard--Floer \cite{OS1, OS2, JM1, JM2, JM3, JM4}, Lagrangian matching \cite{P, P1}, or quilted Floer invariants \cite{WW1, WW2}.

Each homotopy move we use corresponds to a bifurcation in a smooth $1$--
\linebreak parameter family of indefinite fibrations, which may involve a single point (\emph{mono-germ move}) or two/three points (\emph{multi-germ move}), and locally changes the topology of the fibers. Many of these moves have already been studied and employed by several authors, most notably, by Levine, Hatcher--Wagoner, Eliashberg--Mishachev, Lekili, Williams, Gay--Kirby, Behrens--Hayano, and the authors of this article {\cite{B2, B1, BeH, GK2, HW, L, Lev, Sa, Sa2, W, W2}}. As we review mono-germ and multi-germ moves, we will identify a list of  \emph{always-realizable base diagram moves}; namely, local modifications of the singular image on the base (largely corresponding to Reidemeister type moves for fold images), which can be always realized by a homotopy move for indefinite fibrations.

Our first main result is the following:

\begin{thm} \label{firstquotedthm}
Let $X$ be a closed oriented connected $4$--manifold. 
There is an explicit algorithm consisting of sequences of always-realizable mono-germ and multi-germ moves, 
which homotopes any given indefinite fibration $f\colon X \to S^2$ to a simplified indefinite fibration 
$g\colon X \to S^2$, which is fiber-connected, directed, has embedded singular image 
and connected round locus. It suffices to use mono-germ moves flip, unsink,
and cusp merge, and multi-germ moves push, criss-cross braiding, and Reidemeister type moves $\mathrm{R2}^0$, $\mathrm{R2}^1$, 
$\mathrm{R2}_2$, $\mathrm{R3}_2$, $\mathrm{R3}_3$.
\end{thm}

Theorem~\ref{firstquotedthm} will follow from slightly more general procedures, each to strike one of the four  properties we listed earlier, where it will suffice to use even narrower selection of moves; see Theorems~\ref{mainthm1} and~\ref{mainthm2}, Proposition~\ref{prop:locusconnected}.

There is a quantitative and a qualitative advantage to the homotopies we construct. The list of moves in the second part of the theorem compromises less than a fourth of all the elementary moves that naturally arise in bifurcations in generic homotopies of indefinite fibrations (and less than a half of those with only indefinite singularities), along with a special move: the criss-cross braiding move, which is a combination of homotopy moves that are not always-realizable themselves; see Proposition~\ref{lem:umbilic}. In particular, the process does not involve introduction/elimination of a $1$--dimensional round locus component, known as birth/death.
On the other hand, our exclusive use of only always-realizable moves via base diagrams eliminates the need to carry around the very-hard-to-track information for justifying the validity of certain moves.

A necessary condition for an oriented embedded \mbox{$1$--manifold}  $Z$ in a closed oriented $4$--manifold $X$ to be the round locus of an indefinite fibration is that $Z$ is \mbox{null-homologous}, i.e.\ $[Z]=0$ in $H_1(X; \Z)$; see Proposition~\ref{Nullhomologous}. Our second theorem shows that it is also a sufficient condition for realizing $Z$ as the round locus of 
some indefinite fibration, after a homotopy:

\begin{thm} \label{secondquotedthm}
Let $X$ be a closed oriented connected $4$--manifold and $Z$ be a (non-empty) null-homologous 
closed oriented $1$--dimensional submanifold of $X$. There is an \mbox{explicit} algorithm consisting of 
sequences of always-realizable mono-germ and multi-germ moves, which homotopes any given indefinite 
fibration $f \colon X \to S^2$ with non-empty round locus to a fiber-connected, directed broken Lefschetz 
fibration \mbox{$g\colon X \to S^2$} with embedded singular image, whose round locus $Z_g$ coincides with 
$Z$ as oriented  \mbox{$1$--manifolds}. It suffices to use mono-germ moves flip, unsink, cusp merge, and  multi-germ moves push, criss-cross braiding, and Reidemeister type moves $\mathrm{R2}^0$, $\mathrm{R2}^1$,  $\mathrm{R2}_2$, $\mathrm{R3}_2$, $\mathrm{R3}_3$.
\end{thm} 

There are two local models around an indefinite fold circle without cusps; yielding them to be marked as \emph{untwisted} (even) or \emph{twisted} (odd). Akin to the zero locus (the singular locus) of a near-symplectic form \cite{P2}, the number of untwisted components of an indefinite fibration on a closed oriented $4$--manifold $X$ is congruent modulo $2$ to \mbox{$1 + b_1(X) + b^+_2(X)$;}  see Proposition~\ref{untwisted}. The full version of the above theorem, Theorem~\ref{dei}, will show that this necessary condition on the number of untwisted components is also sufficient; we can adjust our algorithm to realize $Z$ as the round locus of an indefinite fibration with prescribed local models.

\smallskip 
Broken Lefschetz fibrations and pencils were introduced by Auroux, Donaldson and Katzarkov in \cite{ADK}, where they proved that every near-symplectic $4$--manifold admits a directed broken Lefschetz pencil with embedded round image using approximately holomorphic geometry.  In \cite{B2}, the first author of this article, using the work of the second author in \cite{Sa}, gave an elementary proof of the existence of broken Lefschetz pencils on near-symplectic $4$--manifolds via singularity theory, and also  established that every closed oriented $4$--manifold admits a broken Lefschetz fibration. A number of alternate proofs of these existence results and their improvements quickly followed: Building on Gay and Kirby's earlier work on \emph{achiral} Lefschetz fibrations \linebreak in \cite{GK1} (where Lefschetz singularities with opposite orientation are allowed), which made extensive use of round \mbox{$2$--handles,} the existence of broken Lefschetz fibrations with \emph{directed, embedded round image} on every closed oriented $4$--manifold was shown in \cite{L, B3, AK}. An alternate singularity theory proof, backed by Cerf theory, was later given in \cite{GK2} to obtain \emph{fiber-connected} broken Lefschetz fibrations. 

Despite this variety however, all the existence proofs so far either made use of geometric results which did not yield to explicit constructions, or fell short of producing broken Lefschetz fibrations/pencils with \emph{simplified} topologies. Handlebody proofs made essential use of Eliashberg's classification of over-twisted contact structures \cite{Eliashberg} and Giroux's correspondence between open books and contact structures \cite{Giroux}, making these procedures non-explicit. Singularity theory proofs came with rather explicit algorithms, but did not succeed in producing a \emph{directed} indefinite fibration with \emph{embedded} round locus. Moreover, none of these works could reproduce an important aspect of Auroux, Donaldson and Katzarkov's pencils: in \cite{ADK}, the authors were able to build their broken Lefschetz pencils so that the \mbox{$1$--dimensional} round locus would coincide with the zero locus (singular locus) of the near-symplectic form. 

By incorporating our new algorithms to simplify the topology of indefinite fibrations, we will improve on the singularity theory approach to derive purely topological and explicit constructions:

\begin{thm} \label{thirdquotedthm} 
Let $X$ be a closed oriented $4$--manifold and $Z$ be a (non-empty) null-homologous closed oriented $1$--dimensional submanifold of $X$. Then, there exists a fiber-connected, directed broken Lefschetz fibration \mbox{$f\colon X \to S^2$} with embedded round image, whose round locus $Z_f$ matches $Z$. Given any generic map from $X$ to $S^2$, such $f$ can be derived from it by an explicit algorithm. If $X$ admits a near-symplectic form $\omega$ with non-empty zero locus $Z_{\omega}$, then there exists a fiber-connected directed broken Lefschetz pencil $f$ on $X$ with  embedded round image, whose round locus $Z_f$ matches $Z_{\omega}$ and $\omega([F])  > 0$ for any fiber $F$ of $f$.
\end{thm}

We will prove the two parts of this theorem in stronger forms in Theorems~\ref{thmA} and~\ref{thmB}, and Corollaries~\ref{simpBLF} and~\ref{simpBLP}.

\vspace{0.1in}
In the last section, we turn to trisections of $4$--manifolds introduced by Gay and Kirby in \cite{GK4}, which are $4$--dimensional analogues of Heegaard splittings of \mbox{$3$--manifolds.}  Just like how one can study Heegaard splittings as certain Morse functions, or as decompositions into two standard handlebodies along with boundary diffeomorphisms, or as Heegaard diagrams, trisections can be studied in three different ways: as certain generic maps (trisected Morse $2$--functions), decompositions into three standard handlebodies along with pairwise partial boundary diffeomorphisms, and trisection diagrams \cite{GK4}. Adopting the first perspective, we simply refer to a trisected Morse $2$--function yielding a trisection decomposition as a \emph{trisection}. This allows us to study trisections as ``almost'' indefinite fibrations, with special topology: they are generic maps to the disk, where a single definite fold circle along the boundary of the disk encloses the image of a fiber-connected, outward-directed indefinite generic map, that can be split into three slices (slicing the disk from a point in the innermost region) so that each sector contains $g'$ fold arcs, $g'-k'$ of which contain a single cusp; see Figure~\ref{fig521}. The preimages of these three slices are the three solid handlebodies $X_i\cong \natural^k (S^1 \times B^3)$ of the decomposition $X=X_1 \cup X_2 \cup X_3$.

We will provide various algorithmic constructions of trisections of $4$--manifolds, based on our simplifications of generic maps through generic homotopy moves.  Our main result establishes a correspondence between simplified broken Lefschetz fibrations and trisections:

\begin{thm} \label{fourthquotedthm} 
Let $X$ be a closed oriented connected $4$--manifold. If there is a genus--$g$ simplified broken Lefschetz fibration $f\colon X \to S^2$ with $k \geq 0$ Lefschetz singularities and $\ell \in \{0, 1\}$ round locus components, then there is a simplified \mbox{$(g',k')$--trisection} of $X$, with 
$(g',k')=(2g+k-\ell+2, 2g-\ell)$. If $X$ admits a simplified $(g', k')$--trisection, then there is a fiber-connected, directed broken  Lefschetz fibration $f\colon X \to S^2$ with embedded round image, which has regular fibers of highest genus $g$ and with $k$ Lefschetz singularities, where $g=g'+2$ and $k=3g'-3k'+4$.
\end{thm}

\noindent Here the genus of a simplified broken Lefschetz fibration is the genus of the highest genus regular fiber, and when $\ell=0$, it is an honest Lefschetz fibration. 

We will prove slightly stronger versions of both directions of the above theorem in Theorem~\ref{BLFtoTS} and Proposition~\ref{TStoBLF}. A by-product of our construction of trisections from simplified broken Lefschetz fibrations is the existence of \emph{simplified trisections} on arbitrary $4$--manifolds, which, combined with our algorithmic constructions of simplified broken Lefschetz fibrations, can be obtained from any given generic map; see Corollary~\ref{STSexistence}. Simplified trisections constitute a subclass of the special Morse \mbox{$2$--functions} yielding trisections, where for a simplified trisection, we in addition have embedded singular image (``no non-trivial Cerf boxes between sectors'') and cusps only appear in triples on innermost circles; see Figure~\ref{fig511} and cf.\ \cite{GK4}. In the extended Example~\ref{exheegaard}, we demonstrate natural examples of simplified trisections and broken Lefschetz fibrations on product $4$--manifolds of the form $S^1 \times Y^3$, which are derived from a Heegaard splitting of the $3$--manifold $Y$.

Finally, building on the above correspondence between simplified broken Lefschetz fibrations and trisections, we construct some interesting families of trisections. In Corollary~\ref{MZquestion}, we show that there are infinitely many homotopy inequivalent $4$--manifolds admitting $(g',k')$--trisections, for each $g' \geq 3$ and $g' -2 \geq k' \geq 1$. In contrast, any $(g',k')$ trisection with $g' < 3$ is standard \cite{MZ}. Furthermore, we show that there are trisections on complex rational surfaces, the three standard sectors of which can be re-glued differently to produce infinitely many homeomoprhic but not diffeomorphic $4$--manifolds; see Corollary~\ref{exotic}.

\vspace{0.2in}
\noindent \textit{Acknowledgments.} The results in this article have been promised for a long while, and several parts were presented by the authors in 2012 Nagano Singularity Conference, 2012 Japan Topology Symposium, and 2013 Bonn Geometry and Topology of \mbox{$4$--manifolds} Conferences. We cordially thank all our colleagues who gently kept pushing us to complete this project. The first author was partially supported by the NSF grants DMS-0906912 and DMS-1510395. The second author has been supported in part by JSPS KAKENHI Grant Numbers JP23244008, JP23654028, JP15K13438, JP16K13754, JP16H03936, JP17H01090.

\vspace{0.1in}
\section{Preliminaries} \label{Sec:Preliminaries}

Here we review the definitions and basic properties of generic maps to surfaces, broken Lefschetz fibrations and pencils, 
along with moves which modify them in smooth $1$--parameter families. All the manifolds and maps are assumed to be smooth.

\subsection{Fold, cusp and Lefschetz singularities; indefinite fibrations} \

Let $f\colon X \to \Sigma$ be a smooth map between compact connected oriented manifolds of dimensions four and two. In the following, we assume that $f^{-1}(\partial \Sigma) = \partial X$ and $f$ is a submersion on a neighborhood of $\partial X$.
Let $y\in \Int{X}$ be a \emph{singular point} of $f$, i.e.\ $\rank{df_y} < 2$. 
The map $f$ is said to have a \emph{fold singularity} at $y$ if there are local coordinates 
around $y$ and $f(y)$ in which the map is given by
\[(t, x_1, x_2, x_3) \mapsto (t, \pm x_1^2 \pm x_2^2 \pm x_3^2) ,\]
and a \emph{cusp singularity} if the map is locally given by
\[(t, x_1, x_2, x_3) \mapsto (t, x_1^3 + t x_1 \pm x_2^2 \pm x_3^2) . \]
A fold or a cusp point $y$ is \emph{definite} if coefficients of all quadratic terms 
in the corresponding local model are of the same sign, \emph{indefinite} otherwise. 
Note that an indefinite cusp is always adjacent to indefinite fold arcs.

A special case of Thom's transversality implies that any smooth map from 
an $n \geq 2$ dimensional space to a surface can be approximated arbitrarily 
well by a map with only fold and cusp singularities \cite{Lev0, T, Wh}. 
Such $f \colon X \to \Sigma$ is called a \emph{generic map} (or an \emph{excellent map}, 
or --more recently-- a \emph{Morse $2$--function}). The \emph{singular locus} $Z_f$ 
of $f$ is assumed to be in $\Int{X}$ and it is a disjoint union of finitely many circles, which are 
composed of finitely many cusp points, and arcs and circles of fold singularities. 
We call $f$ an \emph{indefinite generic} map if all of its fold singularities are indefinite.

On the other hand, the map $f$ is said to have a \emph{Lefschetz singularity} at a point 
$y \in \Int{X}$ if there are orientation preserving local coordinates around $y$ and $f(y)$ 
so that $f$ conforms to the complex local model
\[(z_1, z_2) \mapsto z_1 \, z_2 .\]
A \emph{broken Lefschetz fibration} is a surjective map $f \colon X \to \Sigma$ which 
is only singular along a disjoint union of finitely many Lefschetz critical points and indefinite fold circles. 
A \emph{broken Lefschetz pencil} is then defined for $\Sigma = S^2$, when there is a non-empty, 
finite set $B_f$ of {\emph base points} in $X$, where $f$ conforms to the complex local model
\[(z_1, z_2) \mapsto z_1 / z_2, \]
and $f: X \setminus B_f \to S^2$ has only Lefschetz and indefinite fold singularities \cite{ADK}. 
 
Since the set of generic maps is open and dense in an appropriate mapping space 
endowed with the Whitney $C^{\infty}$ topology, every broken Lefschetz fibration can 
be approximated (and hence homotoped) to a map with only fold and cusp singularities. 
The works of the authors in \cite{Sa} and in \cite{B2} showed that when the base 
$\Sigma=S^2$, one can effectively eliminate the definite fold singularities and cusps in 
order to homotope a generic map to a broken Lefschetz fibration, implying the abundance 
of broken Lefschetz fibrations. Moreover, as shown in \cite{L}, one can trade an indefinite cusp 
point with a Lefschetz singularity, and locally perturb a Lefschetz singularity into a simple indefinite 
singular circle with three cusp points, allowing one to switch between indefinite generic maps and 
broken Lefschetz fibrations in a rather standard way.

With these in mind, we call a smooth surjective map $f\colon X \to \Sigma$ an 
\emph{indefinite fibration}  if it is an indefinite generic map outside of a finite collection of Lefschetz singular points $C_f$. 
For an indefinite fibration $f$, we will call the $1$--dimensional singular locus $Z_f$ the \emph{round locus}, 
and its image $f(Z_f)$ the \emph{round image} of $f$. Note that the restriction of an indefinite fibration to its 
round locus is an immersion except at the cusp points. Its round image is, generically, a collection
of cusped immersed curves on 
$\Sigma$ with transverse double points off the cusp points, which we will assume to be the case hereon. 
We can moreover assume that there is at most one Lefschetz singular point on any fiber, and 
also, Lefschetz critical values and round image are disjoint.

The \emph{base diagram}
of an indefinite fibration is the pair $(\Sigma, f(Z_f \cup C_f))$, 
where the image of any indefinite fold arc or circle is  normally oriented by an arrow, 
which indicates the direction in which the topology of a fiber changes by a $2$--handle attachment 
when crossing over the fold from one side to the other. This means that a generic fiber over the 
region the arrow starts, if connected, has one higher genus than the generic fiber over the region the arrow points into; 
hence the terminology, \emph{higher} and \emph{lower sides} \cite{B1}.

For an indefinite fibration $f\colon X \to \Sigma$, we say that   
\begin{itemize}
\item $f$ is \emph{outward-directed} (resp.\ \emph{inward-directed}), 
if the round image of $f$ is contained in a $2$--disk $D$ in $\Sigma$ such that the complement of a regular value 
$z_0 \in D$ can be non-singularly foliated by arcs oriented from $z_0$ to 
$\partial D$, which intersect the image of each fold arc transversely in its normal direction (resp.\ the opposite direction),
\item $f$ has \emph{embedded round image}, if $f$ is injective on its round locus,
\item $f$ is \emph{fiber-connected}, if every fiber $f^{-1}(z)$, $z \in \Sigma$, is connected.
\end{itemize}
We simply say $f$ is \emph{directed} if it is either outward or inward-directed; when $\Sigma=S^2$, 
one clearly implies the other. 

All these properties are essentially about the round locus and not about Lefschetz critical points. 
Importantly, the topology of an indefinite fibration is much simplified when $f$ satisfies these additional properties. In particular, any fiber-connected $f\colon X \to S^2$ with embedded round image and connected round locus (which implies directed), with all Lefschetz singularities on the higher side, can be captured by simple combinatorial data: an ordered tuple of loops on the highest genera generic fiber (say, the one over $z_0$) \cite{B1, W}. Such $f\colon X \to S^2$ is said to be \emph{simplified}.

\subsection{Near-symplectic structures}  \

A closed $2$--form $\omega$ on an oriented $4$--manifold $X$ is said to be \emph{near-symplectic}, 
if at each point $x \in X$, either $\omega_x^2>0$ (non-degenerate), or $\omega_x = 0$ and the 
intrinsic gradient $\nabla \omega \colon T_x X \to \Lambda^2 (T^* X)$ as a linear map has rank $3$. 
The \emph{zero locus} of $\omega$, i.e.\ the set of points $x \in X$ where $\omega = 0$, is a $1$--dimensional 
embedded submanifold of $X$ denoted by $Z_\omega$. 

Take $\R^{4}$ with coordinates $(t, x_1,x_2, x_3)$ and consider the $2$--form 
\[ \Omega = dt \wedge dQ + *\, (dt \wedge dQ), \]
where $Q(x_1, x_2, x_3)= x_1^{2} + x_{2}^{2} - x_{3}^{2}$ 
and $*$ is the standard Hodge star operator on $\Lambda^{2}\R^{4}$. 
Define two orientation preserving 
affine automorphisms $\sigma_{\pm}$ of $\R^4$ by 
\begin{eqnarray*}
\sigma_+(t, x_1, x_2, x_3) & = & ( t + 2\pi, x_1, x_2, x_3) \ \ \mbox{\rm and} \\
\sigma_- (t, x_1, x_2, x_3) & = & ( t + 2\pi, -x_1, x_2, -x_3).
\end{eqnarray*}  
Restrict $\Omega$ to the product of $\R$ and the unit $3$--ball $D^3$. Each 
$\sigma_{\pm}$ preserves $\Omega$ and the map $(t,x_1, x_2, x_3) \mapsto (t,Q(x_1, x_2, x_3))$, 
and thus, induces a near-symplectic form $\omega_{\pm}$ and an indefinite fold map $f_\pm$ 
on the quotient space $N_{\pm} = (\R \x D^3) / \sigma_{\pm}$. 
As shown in \cite{Ho2}, any near-symplectic form $\omega$ around any component of $Z_\omega$ 
is locally (Lipschitz) equivalent to one of the two local near-symplectic models $(N_{\pm}, \omega_{\pm})$. 
The circles in $Z_\omega$ which admit neighborhoods $(N_+, \omega_+)$ are called \emph{untwisted} 
or of \emph{even} type, and the others \emph{twisted} or of \emph{odd} type. 
Similarly, each component of the round locus of a broken Lefschetz fibration/pencil $f$ is locally equivalent to 
$(N_\pm, f_\pm)$, yielding the same terminology \cite{B1, ADK}. 
Labeling the untwisted components with $0$ and the twisted ones by $1$, we obtain the \emph{twisting data} for 
$Z_\omega$ or $Z_f$.

A $4$--manifold $X$ admits a near-symplectic structure $X$ if and only if \mbox{$b^+(X)>0$} \cite{Ho1, ADK}. (So near-symplectic $4$--manifolds constitute a much larger class than the symplectic ones.) In \cite{ADK}, using approximately holomorphic techniques of Donaldson, the authors proved that for a given near-symplectic form $\omega$ on $X$, there is a directed broken Lefschetz pencil $f$ with embedded round image, such that $Z_f$ coincides with $Z_{\omega}$ with the same twisting data, and such that $\omega(F) > 0$ for any fiber $F$ of $f$. They moreover proved a converse to this result by a Thurston--Gompf construction: if $f$ is a broken Lefschetz pencil on $X$ and there is an $h\in H^{2}(X; \R)$ that evaluates positively on every component of every fiber of $f$, then $X$ admits a near-symplectic form $\omega$, such that $Z_\omega$ coincides with the round locus $Z_f$.

\subsection{Trisections of $4$--manifolds} \label{prelim:trisections} \

A \emph{$(g,k)$--trisection} decomposition, with $g \geq k$, of a closed oriented connected $4$--manifold $X$ is a decomposition $X=X_1 \cup X_2 \cup X_3$, such that: \,(i)\, there is a diffeomorphism $\phi_i\colon X_i \to Z_k$ for each $i = 1,2,3$,  and \,(ii)\,  $\phi_i(X_i \cap X_{i+1}) = Y_{k,g}^-$ and \mbox{$\phi_i(X_i \cap X_{i-1}) = Y_{k,g}^+$} for each $i = 1,2,3$ (mod $3$).  Here $Z_k = \natural^k (S^1 \times B^3)$, $Y_k = \partial Z_k = \sharp^k (S^1 \times S^2)$, and  $Y_k = Y_{k,g}^+ \cup Y_{k,g}^-$ is the standard genus $g$ Heegaard splitting of $Y_k$ obtained by stabilizing the standard genus $k$ Heegaard splitting $g-k$ times. Note that $X_1 \cap X_2 \cap X_3$ is a closed genus--$g$ surface, and $g$ is said to be the \emph{genus of the trisection}.

\begin{figure}[htbp!] 
\begin{center}
\psfrag{a}{(a)}
\psfrag{b}{(b)}
\includegraphics[width=\linewidth,height=0.3\textheight,
keepaspectratio]{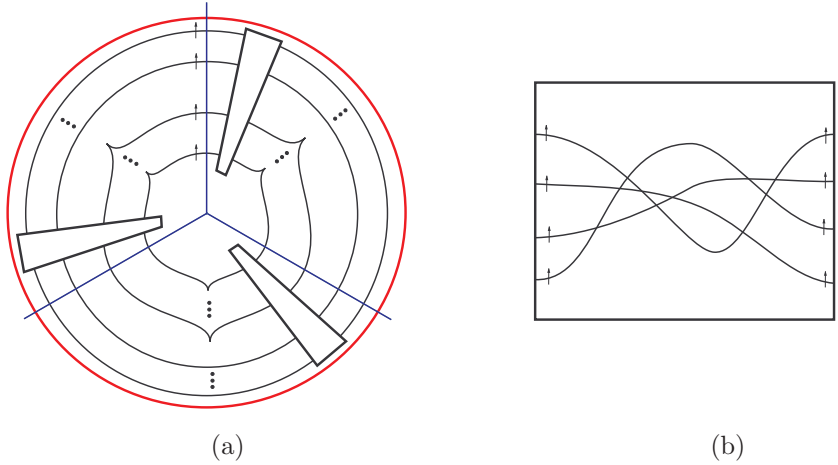}
\caption{(a) Image of a generic map corresponding to a trisection: the outermost circle is the
image of the definite fold circle, and in each box there is an arbitrary Cerf graphic
in the sense of \cite[\S3]{GK4}. The three half lines divide the image
into three parts and their inverse images correspond to $X_1, X_2$ and $X_3$. (b) An example of a Cerf graphic.}
\label{fig521}
\end{center}
\end{figure}

Trisections of $4$--manifolds are introduced by Gay and Kirby in \cite{GK4}. They are to \mbox{$4$--manifolds}, what Heegaard splittings are to \mbox{$3$--manifolds.} Similar to how one can study Heegaard splittings in terms of certain Morse functions, trisections can be studied in terms of certain generic maps,  called \emph{trisected Morse $2$--functions} in \cite{GK4}. In this article, we will adopt this approach, and simply call any trisected Morse $2$--function on $X$ a \emph{trisection} of $X$. This allows us to regard trisections as ``almost'' indefinite directed generic maps. (Bearing in mind that many non-isotopic trisected Morse $2$--functions can yield equivalent trisection decompositions.) Namely, in our language, a trisection corresponds to a generic map to the disk, with an embedded  \emph{definite} fold image enclosing the image of an outward-directed indefinite generic map, with a balanced distribution of cusps to three sectors as in Figure~\ref{fig521}. The total numbers of indefinite fold arcs and indefinite fold arcs without cusps in each sector are $g$ and $k$, respectively, where the arcs with cusps contain a single cusp. Note that the number of fold circles for this special kind of a generic map does \emph{not} need to be equal to $g$, since raising/lowering  them in the Cerf boxes, one may have a fold circle image wrapping around multiple times.

Remarkably, Gay and Kirby showed that just like the Reidemeister--Singer theorem for Heegaard splittings of $3$--manifolds, trisections of $4$--manifolds are unique up to \emph{stabilization}, an operation which can also be described as an introduction of a nested triple of ``wrinkles'' \cite{EM} (or ``eyes''); see \cite{GK4}.

\vspace{0.1in}
\section{Homotopies of indefinite fibrations and base diagram moves} 

In this section, we will discuss certain smooth $1$--parameter families of indefinite fibrations, each of which amounts to a ``move'' from the initial fibration to the final one in the family. Many of these homotopy moves have been studied in varying levels of details by  Levine, Hatcher--Wagoner, Eliashberg--Mishachev, Lekili, Williams, Gay--Kirby, Behrens--Hayano, and the authors of this article in \cite{B2, B1, BeH, GK2, HW, L, Lev, Sa, Sa2, W, W2}. Our goal here is to compile a comprehensive list of moves (with standardized terminology and notation) we can utilize in the rest of the article, for which we will refer to complete arguments in the existing literature,  or provide them if needed. At the end of the section, we will add some \emph{combination} moves to this list, which will play a key role in our topological modifications.

As we are largely interested in moves that will change the general topology of the fibration, 
we will often capture them by studying their singular image. A \emph{base diagram move} is a transition  from $(\Sigma, f_0(Z_{f_0}  \cup C_{f_0}))$ to $(\Sigma, f_1(Z_{f_1} \cup C_{f_1}))$ realized by a smooth  $1$--parameter family $f_t : X \to \Sigma$, $t \in [0, 1]$, such that $f_t$ is an indefinite fibration for each $t$  except for finitely many values in $(0, 1)$. 
(Recall that we assume indefinite fibrations are also injective on their singular locus except possibly at fold double points.) Such a transition essentially happens locally around one point on $\Sigma$;  however, the modification of the map $f_t$ may occur around one, two, or three points in the domain $X$. Following singularity theory conventions, we will call it a \emph{mono-germ move} if the move concerns a single point in $X$, and a \emph{multi-germ move} otherwise.

We will often focus on only parts of the base diagram. Any Lefschetz critical value will be marked by a small cross sign in 
these diagrams. Figure~\ref{Figure: PushCuspsAway} shows two examples. The \emph{unsink} move of \cite{L}, 
which trades an indefinite cusp to a Lefschetz critical point, is a mono-germ move. It clearly changes the singular locus and 
the base diagram, though the isotopy type of the round locus stays the same. The \emph{push} move of \cite{B1}, 
which drags the Lefschetz critical point until its image is on the opposite side of the arrow of the round image, is a 
multi-germ move. In this case, the base diagram changes, but the isotopy type of the singular set does not. 

\begin{figure}[htbp!] 
\psfrag{u}{unsink}
\psfrag{p}{push}
\psfrag{s}{\bang sink}
\psfrag{r}{\quad \bang pull}
\begin{center}
\includegraphics[width=\linewidth,height=0.2\textheight,
keepaspectratio]{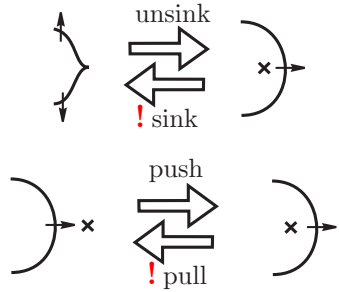}
\caption{Unsink/sink and push/pull moves.}
\label{Figure: PushCuspsAway}
\end{center}
\end{figure}

Both of these are examples of base diagram moves that are \emph{always-realizable} \cite{L, B1}. 
That is, given a local configuration of a base diagram as on the left hand side of Figure~\ref{Figure: PushCuspsAway}, 
we can always find a $1$--parameter family of smooth maps that realizes the relevant base diagram move. On the other 
hand, the \emph{pseudo-inverses} of these two moves, \emph{sink} and \emph{pull} moves in 
Figure~\ref{Figure: PushCuspsAway}, are not always-realizable. A necessary and sufficient condition 
for a sink move is given in terms of  vanishing cycles in \cite{L}. A pull move can be realized if and only if the vanishing 
cycle for the Lefschetz critical point (on a reference fiber on the higher side) and that for the round locus can be chosen to be disjoint.

When there are additional conditions for a move to be realized, we indicate it by an exclamation mark; see e.g.\ 
the pseudo-inverses in Figure~\ref{Figure: PushCuspsAway}. Otherwise, the move is understood to be always-realizable.  
Importantly, \textit{we will only use always-realizable base diagram moves} (some of which will be a combination 
of simpler moves) in this paper. Note that using unsink and push moves, we can always trade an indefinite cusp to 
a Lefschetz critical point, and we can push a Lefschetz critical value across any round image whose arrow is pointing towards it. 
Thus, we will not bother with cusps or Lefschetz critical points in such local diagrams. With this in mind, the collection of moves we cover in the next two subsections can be seen to be sufficient to pass from any given indefinite fibration to another (up to isotopy) by \cite{W}, as we also include  Reidemeister type moves that can appear in bifurcations.

\subsection{Mono-germ moves for indefinite fibrations} \

Figure~\ref{Figure: HomotopyMoves} depicts several well-known mono-germ moves, 
which appear in generic homotopies, and are studied in detail in \cite{Lev, HW, EM, L, W, GK2, BeH}. 

\begin{figure}[htbp!] 
\psfrag{b}{birth}
\psfrag{d}{death}
\psfrag{f}{flip}
\psfrag{u}{unflip}
\psfrag{c}{\bang cusp merge}
\psfrag{fl}{\bang fold merge}
\psfrag{e}{$\emptyset$}
\begin{center}
\includegraphics[width=\linewidth,height=0.2\textheight,
keepaspectratio]{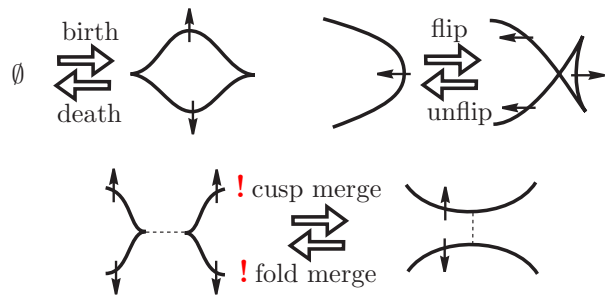}
\caption{Classical mono-germ moves.} 
\label{Figure: HomotopyMoves}
\end{center}
\end{figure}

Note that \emph{birth, death, flip} and \emph{unflip} moves are always-realizable.
For birth and flip moves, see, e.g.\ \cite[Lemmas~3.1 and 3.3]{Sa0}.
For death and unflip moves, \cite[Lemmas~4.7 and 4.8]{GK2} guarantee
that they are realizable under certain conditions. As there are no other singularities appearing
in the local pictures, these conditions are automatically satisfied in our case.
A cusp merge can be performed if and only if there exists a joining curve connecting the pair of cusps
used to eliminate them whose image is depicted by a
dotted line in the diagram; in particular, this move is always-realizable when the fibers over the 
given local disk are connected. 
A necessary and sufficient
condition for a fold merge is that the relevant vanishing cycles intersect transversely at one point 
on a reference fiber over the middle region \cite{L, W, BeH}: in other words, if we take
a vertical oriented line segment in the base depicted in the lower right of
Figure~\ref{Figure: HomotopyMoves}, then its inverse image corresponds to 
a canceling pair of $1$-- and $2$--handles (see also \cite[Lemma~4.6]{GK2}).

Figure~\ref{fig201} shows another example of a mono-germ move from \cite{L} we denote by $W$, called \emph{wrinkling}, which comes from the local perturbation of a Lefschetz critical point. 
(Here the pseudo-inverse $W^{-1}$ may produce a Lefschetz type critical point with the wrong orientation, called an ``achiral Lefschetz critical point''.)

\begin{figure}[htbp!] 
\psfrag{w}{\ \  W}
\psfrag{u}{\bang W$^{-1}$}
\begin{center}
\includegraphics[width=\linewidth,height=0.1\textheight,
keepaspectratio]{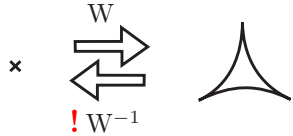}
\caption{Wrinkling move.}
\label{fig201}
\end{center}
\end{figure}

Observe that all these moves, except for flip and unflip, change the set of the singular locus: a birth creates a new circle,  a cusp merge corresponds to a ``band move'' for the set of singular points, etc. Flip and unflip moves, however, both preserve the isotopy type of the singular locus.

\smallskip
\subsection{Multi-germ moves and isotopies of the round locus} \

We now introduce several multi-germ moves, where the round locus simply goes through an isotopy, while the topology of the indefinite fibration changes, at times drastically. Many of these moves have already been studied in, e.g., \cite{Lev, HW, B2, GK1, W2, H3, BeH}. Almost all of these multi-germ moves correspond to the well-known Reidemeister moves II and III for link diagrams in knot theory. However, base diagrams are not simple projections of $1$--dimensional submanifolds, but they are the images of round loci under indefinite fibrations. Furthermore, each fold image has a normal orientation. Therefore,  \textit{there are multiple base diagram moves corresponding to a single Reidemeister type move} (even without any need to involve cusps or Lefschetz critical points in general, 
as we pointed out earlier) . 

To have a uniform notation, Reidemeister II moves will be denoted by 
$\mathrm{R2}$, decorated by subindices $0$, $1$ or $2$, which indicate the number of 
fold arcs with normal arrows pointing into the bounded region, whereas their 
pseudo-inverses (which do not have any bounded regions) will be denoted by the 
same index in the superscript. Reidemeister III moves will be denoted by 
$\mathrm{R3}$, decorated by subindices $0$, $1$, $2$ or $3$, which again 
indicate the number of fold arcs with normal arrows pointing into the bounded region. 
Note that the pseudo-inverse of an $\mathrm{R3}_i$ move is an $\mathrm{R3}_j$ 
move with $i+j=3$; see Figures~\ref{Figure: R2} and \ref{Figure: R3}.

\begin{figure}[htbp] 
\begin{center}
\psfrag{1}{\bang $\mathrm{R2}_0$}
\psfrag{2}{$\mathrm{R2}^0$}
\psfrag{3}{\bang $\mathrm{R2}_1$}
\psfrag{4}{$\mathrm{R2}^1$}
\psfrag{5}{$\mathrm{R2}_2$}
\psfrag{6}{\bang $\mathrm{R2}^2$}
\includegraphics[width=\linewidth,height=0.4\textheight,
keepaspectratio]{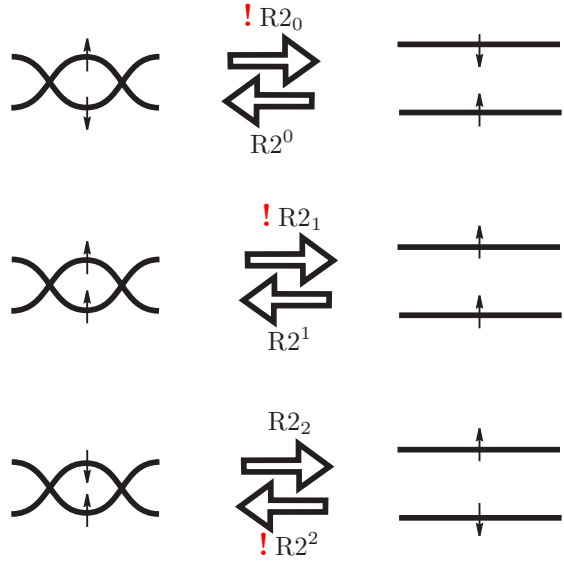}
\caption{Reidemeister II type multi-germ moves.}
\label{Figure: R2}
\end{center}
\end{figure}

\begin{figure}[htbp!] 
\begin{center}
\psfrag{1}{\bang $\mathrm{R3}_0$}
\psfrag{2}{$\mathrm{R3}_3$}
\psfrag{3}{\bang $\mathrm{R3}_1$}
\psfrag{4}{$\mathrm{R3}_2$}
\includegraphics[width=\linewidth,height=0.3\textheight,
keepaspectratio]{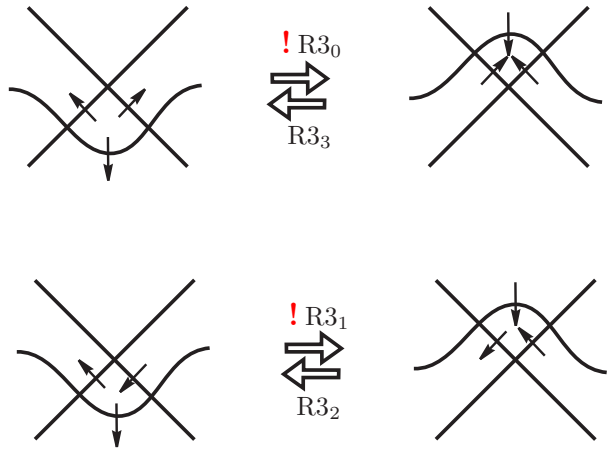}
\caption{Reidemeister III type multi-germ moves.}
\label{Figure: R3}
\end{center}
\end{figure}

Base diagram moves of Reidemeister II type are depicted in Figure~\ref{Figure: R2},  
where only one of each pair is always-realizable. For example, 
for the $\mathrm{R2}_0$ move,
the diagram on the left hand side corresponds, in upward vertical direction, 
to a surface cross interval with a 
$2$--handle and a $1$--handle attached in this order. Horizontally, from left to right, 
the same diagram 
corresponds to the following: first, the order of the two handle attachments 
is switched, then the $2$--handle slides over 
the $1$--handle, and finally the order is switched again. 
So, the $\mathrm{R2}_0$ move is realized if and only if 
the handle-slide trace of the attaching circle of the $2$--handle core 
can be isotoped away from the co-core of 
the $1$--handle.
On the other hand, the diagram on the right corresponds to a handle-slide 
in which the $1$--handle slides 
over the $2$--handle. The trace of the attaching region of a $1$--handle can 
always be isotoped away from a 
$2$--handle, so the move $\mathrm{R2}^0$, the \emph{pseudo-inverse} of 
$\mathrm{R2}_0$, is always-realizable. 
Similar arguments show that the moves $\mathrm{R2}^1$ and $\mathrm{R2}_2$ 
are always-realizable, 
whereas the moves $\mathrm{R2}_1$ and $\mathrm{R2}^2$ are not. (For 
the $\mathrm{R2}_1$ move, see Remark~\ref{rk:R2_1}. The move
$\mathrm{R2}^2$ is not realizable if the attaching circle of the upper
$2$--handle winds along the lower $1$--handle algebraically non-trivially, for
example.)

As to base diagram moves of Reidemeister III type, we have those as depicted in
Figure~\ref{Figure: R3}. We can prove that the move $\mathrm{R3}_3$ is always-realizable 
by an argument similar to that for the move $\mathrm{R2}_2$ above. 
Then, as shown in Figure~\ref{fig251}, the move $\mathrm{R3}_2$ is realized 
as a composition of always-realizable moves $\mathrm{R2}^0$, $\mathrm{R3}_3$ and $\mathrm{R2}_2$. 
So $\mathrm{R3}_2$ move is also always-realizable.
For $\mathrm{R3}_0$ and $\mathrm{R3}_1$ moves, there are necessary \mbox{conditions; cf.\ \cite{W2}.}

\begin{figure}[htbp!] 
\begin{center}
\psfrag{r2}{$\mathrm{R2}^0$}
\psfrag{r3}{$\mathrm{R3}_3$}
\psfrag{r4}{$\mathrm{R2}_2$}
\includegraphics[width=\linewidth,height=0.3\textheight,
keepaspectratio]{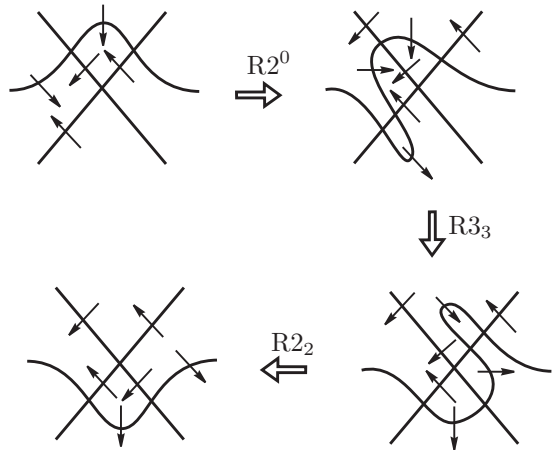}
\caption{Realizing the $\mathrm{R3}_2$ move.}
\label{fig251}
\end{center}
\end{figure}

Lastly, as multi-germ moves involving cusps, we have the \emph{cusp--fold crossing} 
$C$--move and its pseudo-inverse, depicted in Figure~\ref{Figure: C}.  
The $C$--move is always-realizable as seen by an argument in the cusp elimination 
technique of \cite{Lev}; see also \cite[Proposition~2.7]{W2}. The pseudo-inverse
move $C^{-1}$ is not always-realizable, for reasons similar to the case of 
$\mathrm{R2}_0$ or $\mathrm{R2}_1$.

\begin{figure}[htbp]
\centering
\psfrag{c}{$C$}
\psfrag{c1}{\bang $C^{-1}$}
\includegraphics[width=0.95\linewidth,height=0.1\textheight,
keepaspectratio]{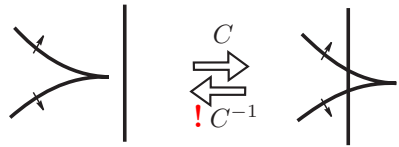}
\caption{Passing a cusp through the round image and its pseudo-inverse. The normal
orientation for the vertical round image component is redundant: both orientations are allowed.}
\vspace{0.2in}
\label{Figure: C}
\end{figure}

\smallskip
\subsection{Some always-realizable combination moves}\label{ss:mixed} \

A combination homotopy move for an indefinite fibration consists of a sequence of mono-germ and multi-germ moves. Our first example 
is the \emph{flip and slip move} of \cite[Fig.~5]{B2}, which can be used to turn an inward-directed circle in $D$ inside out, so it becomes outward-directed. The flip and slip consists of a sequence of always-realizable mono-germ and multi-germ moves shown in Figure~\ref{fig151}. 

\begin{figure}[htbp]
\centering
\psfrag{f}{flip}
\psfrag{r}{$\mathrm{R2}_2$}
\psfrag{s}{unsink}
\includegraphics[width=0.8\linewidth,height=0.1\textheight,
keepaspectratio]{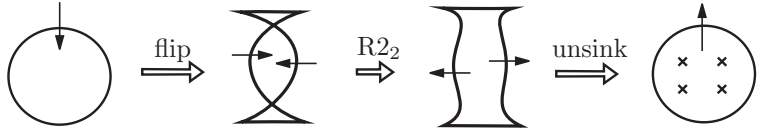}
\caption{Flip and slip.}
\label{fig151}
\end{figure}

Another combination move we will introduce here consists of a sequence of
not necessarily always-realizable moves, but the combined
move itself is always-realizable as a whole, as follows.

\begin{prop}\label{lem:umbilic1}
Let $D$ be a local disk containing the base diagram on the left hand side of 
Figure~\ref{fig431}. Suppose that the fibers over the points
in the region marked with $(\ast)$ are connected.
Then, the \emph{exchange move} depicted in Figure~\ref{fig431} is realizable. 
It is realized by a sequence of two flips, cusp merge, $\mathrm{R3}_1$, 
and unflip moves.
\end{prop}

\begin{figure}[htbp]
\centering
\psfrag{II}{exchange}
\psfrag{a}{$(\ast)$}
\includegraphics[width=0.95\linewidth,height=0.1\textheight,
keepaspectratio]{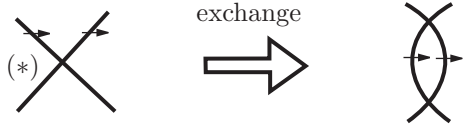}
\caption{Exchange move}
\label{fig431}
\end{figure}

\begin{proof}
Let us identify the disk $D$ with the square
$I \times J$, where $I = J = [-1, 1]$, $I$ corresponds
to the vertical direction downward and $J$
to the horizontal direction from right to left
(see Figure~\ref{fig471}).
We regard $f^{-1}(D) = f^{-1}(I \times J)$
as a $1$--parameter family of $3$--manifolds
$f^{-1}(\{t\} \times J)$, $t \in I$, which
are obtained from $f^{-1}(\{t\} \times [-1, -1+\varepsilon])$,
$0 < \varepsilon << 1$,
by attaching two $1$--handles. By isotoping the handle slides
that may possibly occur while $t \in I$ varies into
intervals outside of $I$, we may assume that
there occurs no handle slide for $t \in I$.
Near $t = 0$, where the crossing of two
fold arc images occurs, the crossing of two
$1$--handles occurs.
Note that the $3$--manifold $f^{-1}(\{t\} \times J)$ obtained by
attaching the two $1$--handles is connected for each $t$
by our assumption.

\begin{figure}[htbp]
\centering
\psfrag{I}{$I$}
\psfrag{J}{$J$}
\psfrag{1}{$1$}
\psfrag{m}{$-1$}
\includegraphics[width=0.95\linewidth,height=0.15\textheight,
keepaspectratio]{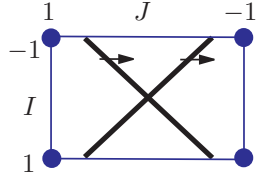}
\caption{Identifying $D$ with $I \times J$.}
\label{fig471}
\end{figure}

Then, we can construct a smooth map
$F \colon V \to I \times J \times [0, 1]$
of a compact $5$--dimensional manifold
with corners $V \cong f^{-1}(I \times J) \times [0, 1]$
with the following properties.
\begin{enumerate}
\item The map $F|_{F^{-1}(I \times J \times
\{0\})} \colon F^{-1}(I \times J \times
\{0\}) \to I \times J \times \{0\}$ coincides with 
\[f|_{f^{-1}(I \times J)} \colon f^{-1}(I \times J)
\to I \times J.\]
\item The singular value set is as depicted in Figure~\ref{fig421}.
\item The map $F$ restricted to $F^{-1}(\partial I \times J \times
[0, 1]) \cong f^{-1}(\partial I \times J) \times [0, 1]$
coincides with the trivial $1$--parameter family of maps 
\[(f|_{f^{-1}(\partial
I \times J)}) \times \mathrm{id}_{[0, 1]} \colon
f^{-1}(\partial I \times J) \times [0, 1] \to
\partial I \times J \times [0, 1].\] 
\end{enumerate}
Note that $F$ can be regarded as a homotopy
of maps $f^{-1}(I \times J) \to I \times J$ \linebreak
parametrized by $[0, 1]$, starting from $f|_{f^{-1}(I \times J)}$
and ending with an indefinite fibration whose
base diagram is as depicted in the right hand side of
Figure~\ref{fig431}.

\begin{figure}[htbp]
\centering
\psfrag{I}{$I$}
\psfrag{J}{$J$}
\psfrag{0}{$[0, 1]$}
\includegraphics[width=0.95\linewidth,height=0.25\textheight,
keepaspectratio]{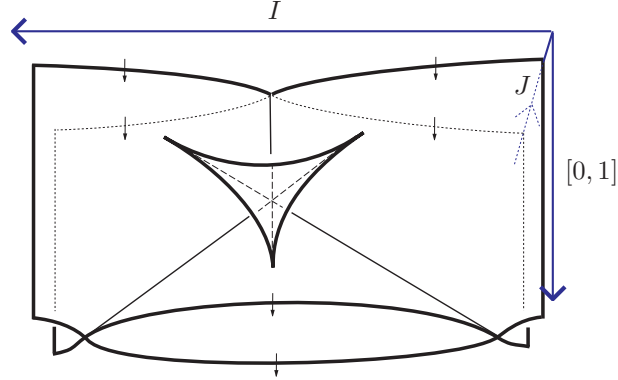}
\caption{The singular value set of $F$ in $I \times J \times [0, 1]$.}
\label{fig421}
\end{figure}

Such a smooth map $F$ can be constructed, for example,
as follows. As indicated in \cite[Chapter~V, \S4]{HW},
to a monkey saddle corresponds a $3$--parameter 
family of functions on a connected $3$--manifold, parametrized by a $3$--ball.
Then the map $F$ corresponds to the southern hemisphere
of the boundary of the parametrizing $3$--ball for
a monkey saddle on an appropriate $3$--manifold.

\begin{figure}[htbp]
\centering
\psfrag{fi}{flips}
\psfrag{c}{cusp}
\psfrag{m}{merge}
\psfrag{III}{$\mathrm{R3}_1$}
\psfrag{f}{unflip}
\psfrag{1}{(a)}
\psfrag{2}{(b)}
\psfrag{3}{(c)}
\psfrag{4}{(d)}
\psfrag{5}{(e)}
\includegraphics[width=\linewidth,height=0.3\textheight,
keepaspectratio]{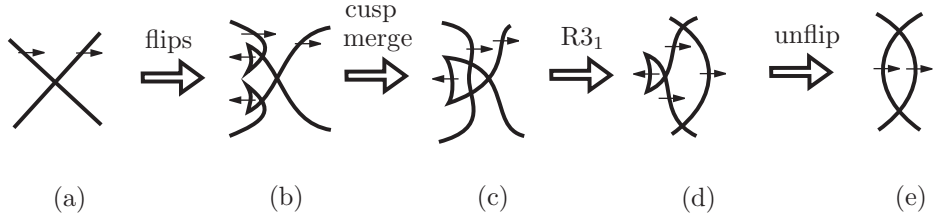}
\caption{Base diagram moves for the exchange move.}
\label{fig441}
\end{figure}

This can also be seen by the sequence of base diagram moves as depicted in Figure~\ref{fig441}. 
The transition from (a) to (b) is realized
by two flips. The transition from (b) to (c)
is realized by a cusp merge, which is realizable
as the relevant fibers are all connected by our assumption.
The transition from (c) to (d) is realized by
$\mathrm{R3}_1$ move, which is not always-realizable.
However, in our case, the $1$--handle that is
attached over the two crossing $1$--handles does
not slide over them, and therefore the $\mathrm{R3}_1$ move
is realizable. Finally, the transition from
(d) to (e) is realized by an unflip.

It might be good to remind that here we are constructing 
a smooth $1$--parameter family of maps 
(from a $4$--manifold to a surface) which starts from 
a given indefinite fibration with a base diagram as 
in Figure~\ref{fig441} (a), and which ends up with a certain (not given!) 
indefinite fibration with a base diagram 
as in Figure~\ref{fig441} (e). That is, our goal is to show that 
we can appropriately choose indefinite fibrations with 
base diagrams as in Figure~\ref{fig441} (b)--(e); we are not 
trying to reconstruct \emph{any} possible transition for 
such a sequence of base diagram moves. In particular, we can choose 
these maps so that the handle-slides 
are arranged as argued in the previous paragraphs.

This completes the proof of the proposition.
\end{proof}

Our next combination move is an immediate corollary of the above proposition:

\begin{prop}[Criss-cross braiding] \label{lem:umbilic} 
Let $D$ be a local disk containing the base diagram on the left hand side of 
Figure~\ref{fig:crisscross}. 
Suppose that the fibers over the points
in the region marked with $(\ast)$ are connected. 
Then, the criss-cross braiding depicted in Figure~\ref{fig:crisscross} 
is realizable. 
It is realized by a sequence of $\mathrm{R2}^1$, two flips, cusp merge, $\mathrm{R3}_1$, 
and unflip moves.
\end{prop}

\begin{figure}[htbp]
\centering
\psfrag{cc}{criss-cross}
\psfrag{a}{$(\ast)$}
\includegraphics[width=0.95\linewidth,height=0.1\textheight,
keepaspectratio]{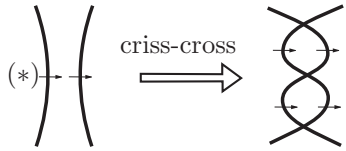}
\caption{\emph{Criss-cross braiding} that can be always-realized.}
\label{fig:crisscross}
\end{figure}

\begin{proof}
We first apply $\mathrm{R2}^1$ move. Then, by our assumption,
we can apply 
Proposition~\ref{lem:umbilic1} to one of the crossings to
get the desired base diagram.
\end{proof}

\smallskip
Taking a closer look at our proofs of above propositions, we can identify some necessary conditions for the Reidemeister type moves we have not identified as always-realizable. We discuss these in the next several remarks.

\begin{rk}\label{rk:homological-constraint}
In the proof of Proposition~\ref{lem:umbilic1}, we
constructed a smooth map 
\[F : V \to I \times J \times [0, 1].\]
For $(s_1, s_2) \in I \times [0, 1]$, set 
$$F_{s_1, s_2}
= p_J \circ F|_{F^{-1}(\{s_1\} \times J \times \{s_2\})} \colon
F^{-1}(\{s_1\} \times J \times \{s_2\}) \to J,$$
where $p_J \colon I \times J \times [0, 1] \to J$ is the
projection to the second factor. This is a family of functions on a $3$--manifold parametrized
by $I \times [0, 1]$. If $(s_1, s_2) \in \partial (I \times [0, 1])$, then $F_{s_1, s_2}$ has exactly two critical points of index $1$.
Observing the monkey saddle point carefully, we see that as $(s_1, s_2) \in \partial (I \times [0, 1])$ varies, we have the handle slides as depicted in Figure~\ref{fig451}.

\begin{figure}[htbp]
\centering
\psfrag{h}{$2$--handle}
\includegraphics[width=\linewidth,height=0.4\textheight,
keepaspectratio]{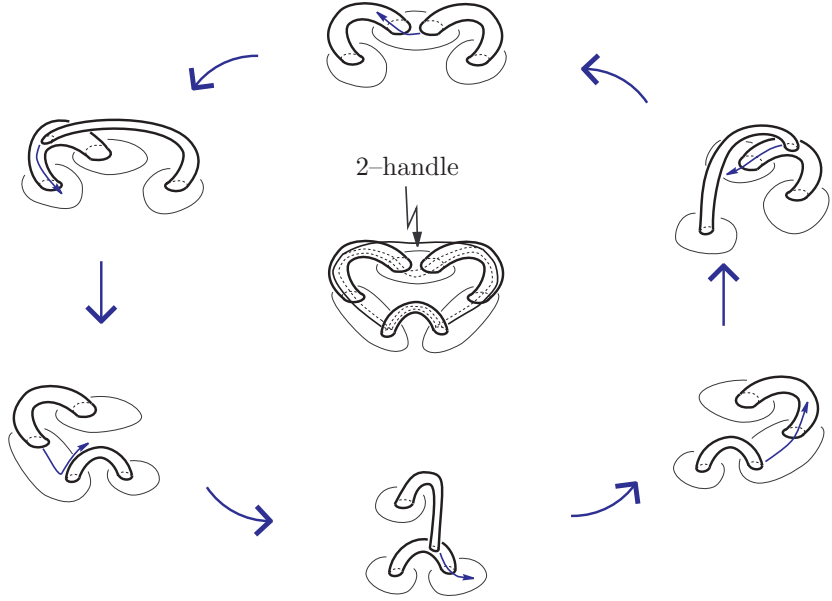}
\caption{Handle slides for $F_{s_1, s_2}$, $(s_1, s_2) \in
\partial (I \times [0, 1])$. The central figure corresponds
to $F_{0, 1/2}$ which has exactly
three critical points of index $1$ with the same value, together
with a critical point of index $2$ over them.
}
\label{fig451}
\end{figure}

Let us examine more carefully the handle slides involving an
exchange move. For this, let us investigate
the homological behavior of the handle slides for 
the base diagrams on both sides of Figure~\ref{fig431}.
For $t_1 = -1 \in I$, let $\alpha$ (or $\beta$) denote the 1st homology
class corresponding to the upper (resp.\ lower) $1$--handle
(see Figure~\ref{fig461}). (Precisely speaking, 
we fix an orientation of each $1$--handle, and then it represents
an element of $H_1(f^{-1}(\{t_1\} \times J), f^{-1}(\{t_1\} \times \{-1\}); \Z)
\cong \Z \oplus \Z$.)
Let us first consider the
base diagram on the left hand side of Figure~\ref{fig431}.
We consider the four parameter values of $I$ indicated in Figure~\ref{fig461}. 
We assume that for $t \in [t_1, t_2]$ (or $t \in [t_3, t_4]$), 
the upper $1$--handle
slides over the lower $1$--handle homologically $p$ times
(resp.\ $q$ times). We further assume that for $t \in [t_2, t_3]$
no handle slides occur. Then, a simple calculation shows that
for the level $t_4 = 1 \in I$, the
upper $1$--handle represents the homology class
$\beta + q(\alpha + p \beta)
= q \alpha + (pq+1) \beta$, while the lower one represents
$\alpha + p \beta$.

\begin{figure}[htbp]
\centering
\psfrag{a}{$\alpha$}
\psfrag{b}{$\beta$}
\psfrag{k}{$k$}
\psfrag{l}{$\ell$}
\psfrag{m}{$m$}
\psfrag{p}{$p$}
\psfrag{q}{$q$}
\psfrag{1}{$t_1$}
\psfrag{2}{$t_2$}
\psfrag{3}{$t_3$}
\psfrag{4}{$t_4$}
\psfrag{5}{$t_5$}
\psfrag{6}{$t_6$}
\psfrag{7}{$t_7$}
\psfrag{8}{$t_8$}
\psfrag{9}{$t_9$}
\psfrag{10}{$t_{10}$}
\includegraphics[width=\linewidth,height=0.2\textheight,
keepaspectratio]{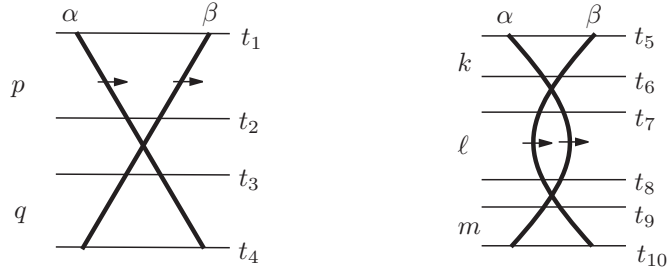}
\caption{As $t \in I$ varies, the upper $1$--handle
slides over the lower one, and the corresponding homological
``winding numbers'' are denoted by $p, q$ on the left hand side,
while they are denoted by $k, \ell, m$ on the right hand side.}
\label{fig461}
\end{figure}

Let us play the same game for the base diagram on the right hand side
of Figure~\ref{fig431}.
Then, using the notations indicated on the
right hand side of Figure~\ref{fig461},
we see that for the level $t_{10} = 1 \in I$,
the upper $1$-handle represents 
$\alpha + k \beta + m(\beta + \ell(\alpha + k \beta))
= (\ell m+1) \alpha + (k + m + k \ell m) \beta$,
while the lower one represents
$\beta + \ell (\alpha + k \beta) = \ell \alpha
+ (k \ell + 1) \beta$.

Therefore, if the transition between the both sides
of Figure~\ref{fig461} is realized by homotopy, then we must have
\begin{eqnarray*}
q \alpha + (pq+1) \beta & = &  
\varepsilon_1((\ell m+1) \alpha + (k + m + k \ell m) \beta), \\
\alpha + p \beta & = & \varepsilon_2(\ell \alpha + (k \ell + 1) \beta)
\end{eqnarray*}
for some $\varepsilon_1, \varepsilon_2 \in \{-1, +1\}$.
By a straightforward calculation, we see that
these hold if and only if we have
$$\ell = \varepsilon_2, \, \varepsilon_1 \varepsilon_2 = -1, \,
p = k + \varepsilon_2, \, q = -m + \varepsilon_1.$$
In particular, we see that for $t \in [t_7, t_8]$
in the figure on the right hand side, the upper $1$--handle
must slide over the lower one homologically $\pm 1$ time.

In the case of the proof of Proposition~\ref{lem:umbilic1}, we have $p=q=0$,
so $k=m=\varepsilon_1$ and $\ell = - \varepsilon_1$.
This conforms to the proof of the lemma and Figure~\ref{fig451}.
\end{rk}

\begin{rk}\label{rk:R2_1}
The simple homological calculation in Remark~\ref{rk:homological-constraint}
also implies the following observation. Let us assume that the $\mathrm{R2}_1$ 
move is realized. Then, the base diagram on the right hand side
of Figure~\ref{fig461} must be homotopic to two parallel strands.
Suppose that for this latter base diagram, the
upper $1$--handle slides over the lower one homologically $r$ times.
Then, by an argument similar to the above, we must have
\begin{eqnarray*}
\alpha + r \beta & = &  
\varepsilon_1((\ell m+1) \alpha + (k + m + k \ell m) \beta), \\
\beta & = & \varepsilon_2(\ell \alpha + (k \ell + 1) \beta)
\end{eqnarray*}
for some $\varepsilon_1, \varepsilon_2 \in \{-1, +1\}$.
By a straightforward calculation, we see that these equalities hold if and only if
$\ell = 0, \, r = k+m, \, \varepsilon_1 = \varepsilon_2 = 1.$
This means that the $\mathrm{R2}_1$ move is not always-realizable:
for the realization, it is necessary that for $t \in [t_7, t_8]$, the
upper $1$--handle should not slide over the lower one at least homologically.
\end{rk}

\begin{rk}\label{rk:williams}
Let us assume that the base diagram on the left hand side
of Figure~\ref{fig461} is transformed to a pair
of vertical strands by homotopy. Then, we see easily that exactly the same
argument as above leads to a contradiction. 
\end{rk}

\vspace{0.1in}
\section{Simplifying the topology of indefinite fibrations} \label{Sec:Embedded}

In this section, we will give explicit algorithms for homotoping an indefinite fibration 
to a directed indefinite fibration, and in turn, to a directed indefinite fibration with embedded round image. Our algorithms will use base diagram moves which are always-realizable. These will consist of flip, unsink, push, Reidemeister type multi-germ moves gathered in Figure~\ref{fig32-3}, and the additional \emph{criss-cross braiding} move given in Proposition~\ref{lem:umbilic}. We will also give similar algorithms for homotoping a directed indefinite fibration with embedded round image to one which is also fiber-connected and has connected round locus. 

\begin{figure}[htbp!]
\psfrag{r2u0}{$\mathrm{R2}^0$}
\psfrag{r2u1}{$\mathrm{R2}^1$}
\psfrag{r2l2}{$\mathrm{R2}_2$}
\psfrag{r3l3}{$\mathrm{R3}_3$}
\psfrag{r3l2}{$\mathrm{R3}_2$}
\centering
\includegraphics[width=\linewidth,height=0.39\textheight,
keepaspectratio]{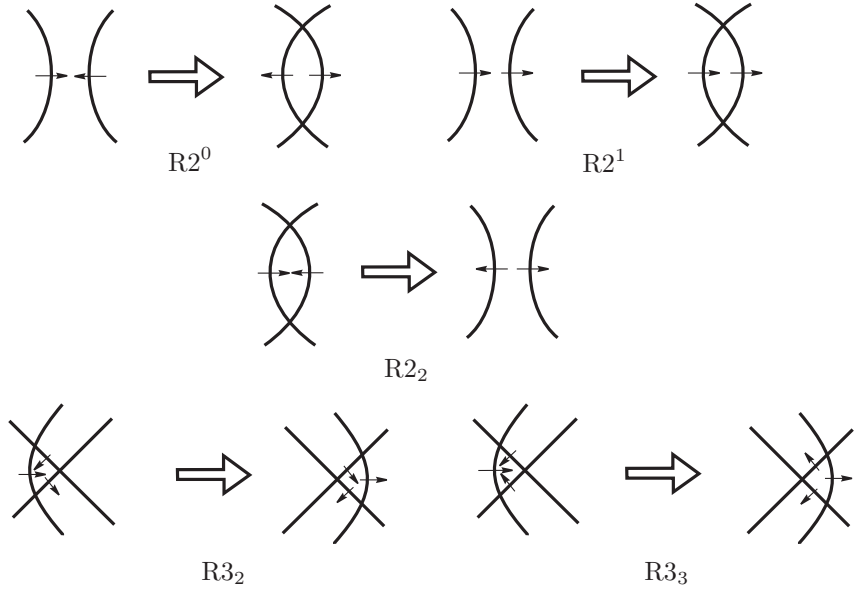}
\caption{Several multi-germ moves that are always-realizable.}
\label{fig32-3}
\end{figure}

Note that for any one of the multi-germ moves in Figure~\ref{fig32-3}, if the fibers over the given base diagram  are all connected, then so are the fibers over the base diagram we get after applying the move.  The only concern here can be for fibers over bounded regions formed after the move. However, in each case,  one can reach to these regions from a region on the periphery by ``going against the arrows'', i.e.\ a fiber here is  obtained by adding only $1$--handles to a connected fiber on the peripheral region.

\smallskip
\subsection{Immersed directed round image} \

We will now prove:

\begin{thm} \label{mainthm1}
There exists a finite algorithm consisting of sequences of flip, unsink, push and Reidemeister type moves $\mathrm{R2}^0$, $\mathrm{R2}^1$, 
$\mathrm{R2}_2$, $\mathrm{R3}_2$, $\mathrm{R3}_3$, which homotopes any given indefinite fibration $f\colon X \to D^2$ to an  outward-directed indefinite fibration $g\colon X \to D^2$. When $f$ is fiber-connected, so is the resulting outward-directed indefinite fibration $g$.
\end{thm}

Note that all the moves mentioned in the above theorem are always-realizable.

\begin{proof}
Applying unsink moves, we may assume that $f$ has no cusps. Take an embedded annulus $A = [0, 1] \times S^1 $ in $D^2$ 
which contains $f(Z_f)$ in its interior, and let $\pi \colon A \to S^1$ be the natural projection to the second factor. 
We may assume that $\pi \circ f|_{Z_f}$ is a Morse function and that the $\pi$--values of the crossings of $f(Z_f)$ are 
different than the critical values of $\pi \circ f|_{Z_f}$ in $S^1$. Let $t_1, t_2, \ldots, t_s \in S^1$ be the critical values of 
$\pi \circ f|_{Z_f}$, located in this order with respect to a fixed orientation of $S^1$, and consider the ``zones'' 
$A_i = [0, 1] \times  [t_i, t_{i+1}]  \subset A$, $i = 1, 2, \ldots, s$, where $[t_i, t_{i+1}] \subset S^1$ is the directed 
arc from $t_i$ to $t_{i+1}$, and $t_{s+1} = t_1$. Then, $f(Z_f) \cap \Int{A_i}$ consists of a finite number of 
embedded and co-oriented vertical open arcs, where ``vertical'' means that $\pi$ restricted to each open arc is a submersion. 
We say that such an arc is \emph{positive} (resp.\ \emph{negative}) if its co-orientation is consistent with the positive 
(resp.\ negative) direction of the $[0, 1]$--factor of $A_i$.  If necessary, applying local isotopies as in Figure~\ref{fig101} 
in advance,  we may assume that positive arcs do not intersect each other, in the expense of subdividing $S^1$ further.

\begin{figure}[htbp!]
\centering
\includegraphics[width=0.5\linewidth,height=0.15\textheight]{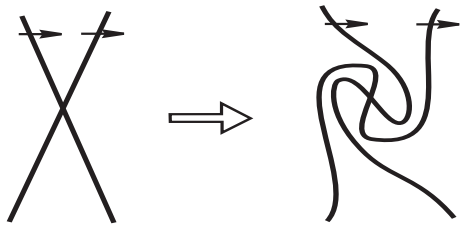}
\caption{Turning an intersection of positive arcs into that of negative arcs.}
\label{fig101}
\end{figure}

Next, we will apply base diagram moves to $f$ in such a way that, for the resulting indefinite fibration, all the arcs in each zone $A_i$ are negative. We will continue to denote the resulting indefinite fibration by the same letter $f$, for simplicity. The following conditions will be preserved after each modification for each $i=1, 2, \ldots, s$:
\begin{itemize}
\item[(1)] Each arc in $f(Z_f) \cap A_i$ is vertical,
\item[(2)] Positive arcs in $f(Z_f) \cap A_i$ do not intersect each other.
\end{itemize}

\begin{figure}[tb]
\psfrag{ai}{$A_i$}
\psfrag{a1}{$\alpha_1$}
\psfrag{a2}{$\alpha_2$}
\psfrag{al}{$\alpha_\ell$}
\psfrag{-1}{$0$}
\psfrag{1}{$1$}
\centering
\includegraphics[width=0.8\linewidth,height=0.35\textheight]{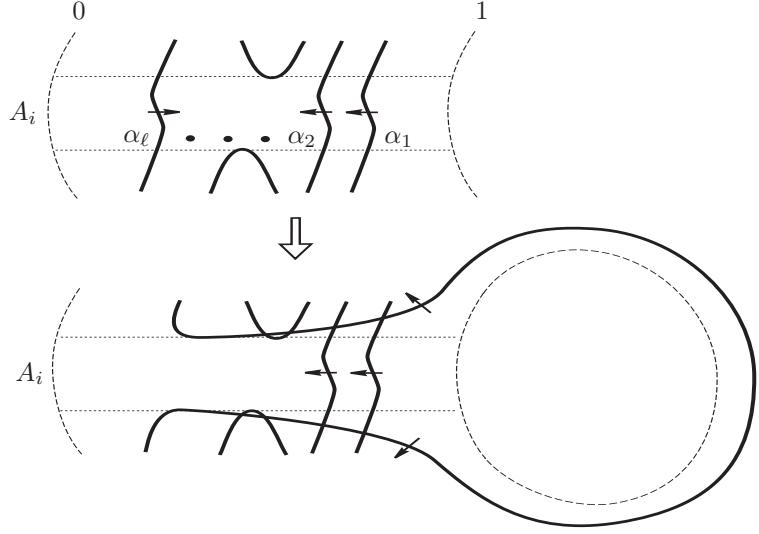}
\caption{Moves that eliminate a positive arc: Part 1.}
\label{fig51}
\end{figure}

Fix an index $i$. If $A_i$ does not contain any crossing of $f(Z_f)$, then we can name the arcs of $f(Z_f) \cap \Int{A_i}$ as $\alpha_1, \alpha_2, \ldots, \alpha_r$ such that $\alpha_1$ is situated
in the rightmost position, and then $\alpha_2$ is next to it on its left hand side, and so on.
Suppose that $\alpha_1, \alpha_2, \ldots, \alpha_{\ell-1}$ are negative and $\alpha_\ell$ is positive for some $1 \leq \ell \leq r$. When the end points of $\alpha_\ell$ are not
$f$--images of critical points of $\pi \circ f|_{Z_f}$, we apply the always-realizable multi-germ moves in Figure~\ref{fig32-3} together with pushes so that the move as depicted in Figure~\ref{fig51} is realized. As a result, the total number of positive arcs (in \emph{any} $A_i$) decreases. During this procedure, the arc in question goes out of the annulus $A$ temporarily, but in the end it is embedded in $A$ \emph{with negative co-orientation}. Note that the resulting base diagram still satisfies the above conditions (1) and (2); the number of arcs in $A_j$, $j \neq i$, increases, but they are all negative.

When an end point of $\alpha_\ell$ is the $f$--image of a critical point of $\pi \circ f|_{Z_f}$, we apply the sequences of 
moves as depicted in Figure~\ref{fig61} or in Figure~\ref{fig71}, which consist of the always-realizable multi-germ 
moves in Figure~\ref{fig32-3} and pushes, as well as a flip and a pair of unsinks in the latter case.

\begin{figure}[tb]
\centering
\includegraphics[width=0.85\linewidth,height=0.13\textheight,
keepaspectratio]{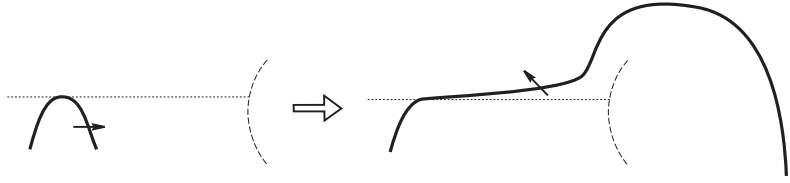}
\caption{Moves that eliminate a positive arc: Part 2.}
\label{fig61}
\end{figure}

\begin{figure}[htbp!]
\psfrag{a}{$\alpha_{\ell-1}$}
\psfrag{b}{$\alpha_{\ell}$}
\psfrag{f}{flip}
\psfrag{u}{unsinks and multi-germ moves}
\centering
\includegraphics[width=0.85\linewidth,height=0.3\textheight,
keepaspectratio]{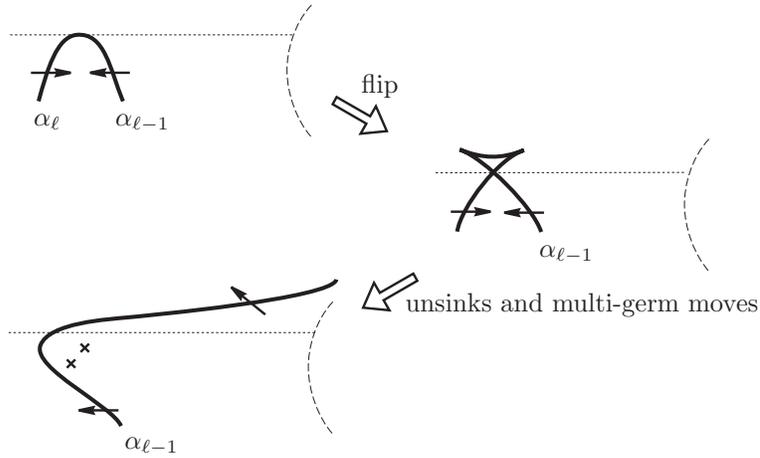}
\caption{Moves that eliminate a positive arc: Part 3.}
\label{fig71}
\end{figure}

Now consider the case when $A_i$ contains crossings of $f(Z_f)$. Let $\alpha_1, \alpha_2, \ldots, \alpha_r$ be the (open) vertical arcs of $f(Z_f) \cap \Int{A_i}$. If they are all negative, then we have nothing to do. Suppose some are positive.  Say $\alpha_\ell$ is the rightmost positive vertical arc; i.e.\ there are no positive vertical arc components in the right hand side region of $\Int{A_i} \setminus \alpha_\ell$.  (Note that if there are positive vertical arcs, then such $\alpha_\ell$ exists, since we made sure that positive arcs do not intersect.)  On $\alpha_\ell$, there may be intersections with negative arcs. If there is no crossing in the right hand side region, then we can apply moves similar to those used above to decrease the number of positive arcs. Otherwise, all the crossings in the right hand side region of 
$\Int{A_i }\setminus \alpha_\ell$ involve only negative arcs. When we move $\alpha_\ell$ to the right using moves as described above, it may encounter such a crossing of negative arcs. In that case, we get a triangular region such that one edge is on $\alpha_\ell$ and the other two are on negative arcs. We may assume that the edges are line segments, which are never horizontal.
If the ``height'' of the vertex of this triangle that is not on $\alpha_\ell$ is between the heights of the other two vertices, then we have a situation as described in the left hand side picture of Figure~\ref{fig111}.
In this case, we can apply the move $\mathrm{R3}_3$ to decrease
the number of crossings in the right hand side region. If that vertex is lower (or higher) than the other two vertices, then we get  the left hand side picture of Figure~\ref{fig121}, and we can apply the move $\mathrm{R3}_2$.

Then, by the same argument as above, we can finally eliminate the positive arc $\alpha_\ell$.
Note that after the moves, both  conditions (1) and (2) are maintained.

\begin{figure}[tb]
\psfrag{al}{$\alpha_{\ell}$}
\psfrag{3}{$\mathrm{R}3_3$}
\centering
\includegraphics[width=0.45\linewidth,height=0.15\textheight]{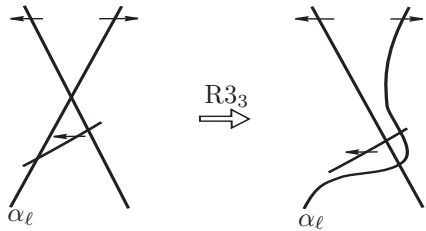}
\caption{Moves that eliminate a positive arc in $A_i$: Part 4.}
\label{fig111}
\end{figure}

\begin{figure}[tb]
\psfrag{al}{$\alpha_{\ell}$}
\psfrag{3}{$\mathrm{R}3_2$}
\centering
\includegraphics[width=0.45\linewidth,height=0.15\textheight]{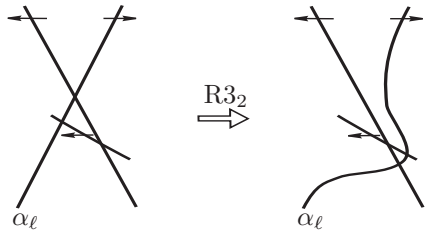}
\caption{Moves that eliminate a positive arc in $A_i$: Part 5.}
\label{fig121}
\end{figure}

Applying these procedures to the positive arcs of $f(Z_f) \cap \Int{A_i}$ from right to left,
we can eliminate all the positive arcs in $\Int{A_i}$, for each $i= 1, 2, \ldots, s$.  
Finally, we get an indefinite fibration $g$ where all arcs in $g(Z_g) \cap \Int{A_i}$ are negative for all $i = 1, 2, \ldots, s$, 
which means that the resulting indefinite fibration $g$ is outward-directed.

Note that our whole algorithm was performed away from the boundary of the base disk $D^2$, 
where the resulting round image is directed outwards. So if a regular fiber $F$ of $f$ over $\partial D^2$ 
was connected, then \emph{all} fibers of $g$ are connected: they are derived by fiberwise $1$--handle 
attachments from $F$. In particular, $g$ is fiber-connected \mbox{if $f$ is.}
\end{proof}

\smallskip

\subsection{Embedded directed round image}\label{EmbeddedSection} \

In this subsection, we assume that the closed orientable $4$--manifold
$X$ is connected, and we prove:

\begin{thm} \label{mainthm2}
There exists a finite algorithm consisting of sequences of flip, unsink, push, criss-cross braiding and Reidemeister type moves $\mathrm{R2}^0$,  $\mathrm{R2}^1$, $\mathrm{R2}_2$, $\mathrm{R3}_2$, $\mathrm{R3}_3$, which homotopes any given directed indefinite fibration $f\colon X \to S^2$ to a fiber-connected and directed indefinite fibration \mbox{$g\colon X \to S^2$} with embedded round image. 
\end{thm}

\begin{proof}
Below, as we modify the map $f$ through homotopies, we will keep denoting the resulting map with the same letter. 

Since $f$ has immersed directed image, we can view the round image $f(Z_f)$ to be braided around, say the north pole, 
directed towards it. That is, we regard $f(Z_f)$ as the closure of a virtual braid on $m$ strands. 
We will call it the \emph{base braid} for $f$. For convenience, we will think of  $f(Z_f)$ as the union of a base braid, 
given by a virtual braid diagram as in Figure~\ref{fig40}, and a trivial braid, where the latter is juxtapositioned to 
the former to recapture the round image $f(Z_f)$ as the braid closure. 

\begin{figure}[htbp]
\centering
\includegraphics[width=0.6\linewidth,height=0.12\textheight,
keepaspectratio]{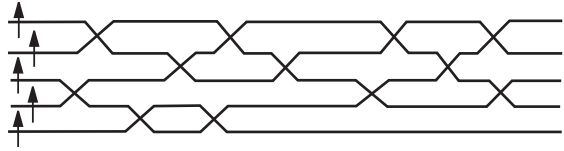}
\caption{An example of a base braid diagram on $5$ strands. The round image is its closure, 
which is obtained by end summing the base braid with a trivial braid on $5$ strands.}
\label{fig40}
\end{figure}

First, we will show that the fibers over the southernmost region of $S^2 \setminus f(Z_f)$
are connected. Let $A\cong S^1 \times [-1, 1]$ be an annular neighborhood of $f(Z_f)$, where the interval $[-1, 1]$ points north. 
So $S^2 \setminus A = N_{\text{NP}} \sqcup N_{\text{SP}}$, where $N_{\text{NP}}$ and $N_{\text{SP}}$  
denote open $2$--disk neighborhoods of the north and south poles, respectively. 
The images of all the Lefschetz critical points are contained in $N_{\text{SP}}$. We can assume that the map 
$\pi_{S^1} \circ f|_{Z_f} \colon Z_f \to S^1$ is a submersion, where, under the identification 
$A \cong S^1 \times [-1, 1]$, \mbox{$\pi_{S^1}\colon S^1 \times [-1, 1] \to S^1$} is the projection to 
the first factor. Then, $X$ is decomposed into three compact $4$--manifolds 
{$f^{-1}(\overline{N}_{\text{NP}})$, $f^{-1}(\overline{N}_{\text{SP}})$} 
and $f^{-1}(A)$. Note that 
{$f^{-1}(\overline{N}_{\text{NP}})$} is a trivial surface bundle over 
{$\overline{N}_{\text{NP}}$},
while $f^{-1}(A)$ is a fiber bundle over $S^1$ with fiber $Y = f^{-1}(\text{pt}) \times [-1, 1])$, 
where $Y$ is a $3$--manifold obtained from 
$f^{-1}(\text{pt} \times [1-\varepsilon, 1])$, $0 < \varepsilon <<1$, by attaching $m$ $1$--handles.
Suppose the fibers are disconnected over $N_{\text{SP}}$. 
The $3$--manifold {$f^{-1}(\partial \overline{N}_{\text{SP}})$}
is a fiber bundle over $S^1$ with fiber surface, say $S$,
with monodromy generated by Dehn twists.
Therefore, the monodromy diffeomorphism preserves each 
connected component of $S$.
Then, for the $Y$--bundle over $S^1$, $f^{-1}(A)$, the monodromy also preserves 
each connected component of $Y$. Therefore, by attaching 
{$f^{-1}(\overline{N}_{\text{NP}})$ and $f^{-1}(\overline{N}_{\text{SP}})$} 
to $f^{-1}(A)$, we get
a disconnected $4$--manifold, which is a contradiction. 
Thus, the fibers over $N_{\text{SP}}$ are necessarily connected.

We can now give our algorithm to prove the theorem. If $m=0$ or $1$, then there is nothing to do, so we assume there are 
$m \geq 2$ strands. 

\smallskip
\noindent \textit{\underline{Step 1:} }
As we have shown above, the fibers over the points in the 
southernmost region of $S^2
\setminus f(Z_f)$ are connected.
Using the $\mathrm{R2}^1$ moves,
we can locally pull down a pair 
of parallel strands towards the south pole, so that we get
a pair of parallel strands such that the lower one is adjacent to
the southernmost region (see the middle diagram of
Figure~\ref{fig481}). Then, by a
criss-cross braiding move of Proposition~\ref{lem:umbilic}, we can locally replace the pair
of parallel strands with a pair of strands that have three mutual crossings 
(see the rightmost diagram of Figure~\ref{fig481}), while modifying the 
round locus above it. This modification 
acts as a transposition on the $m$ points the virtual braid is moving around. 
Since the symmetric group of $m$ points is 
generated by transpositions, 
by adding enough crossings, the base braid for the new round image $f(Z_f)$ 
becomes a \emph{pure} virtual braid. The new round locus $Z_f$ has exactly 
$m$ connected components.

\begin{figure}[htbp!]
\centering
\includegraphics[width=\linewidth,height=0.35\textheight,
keepaspectratio]{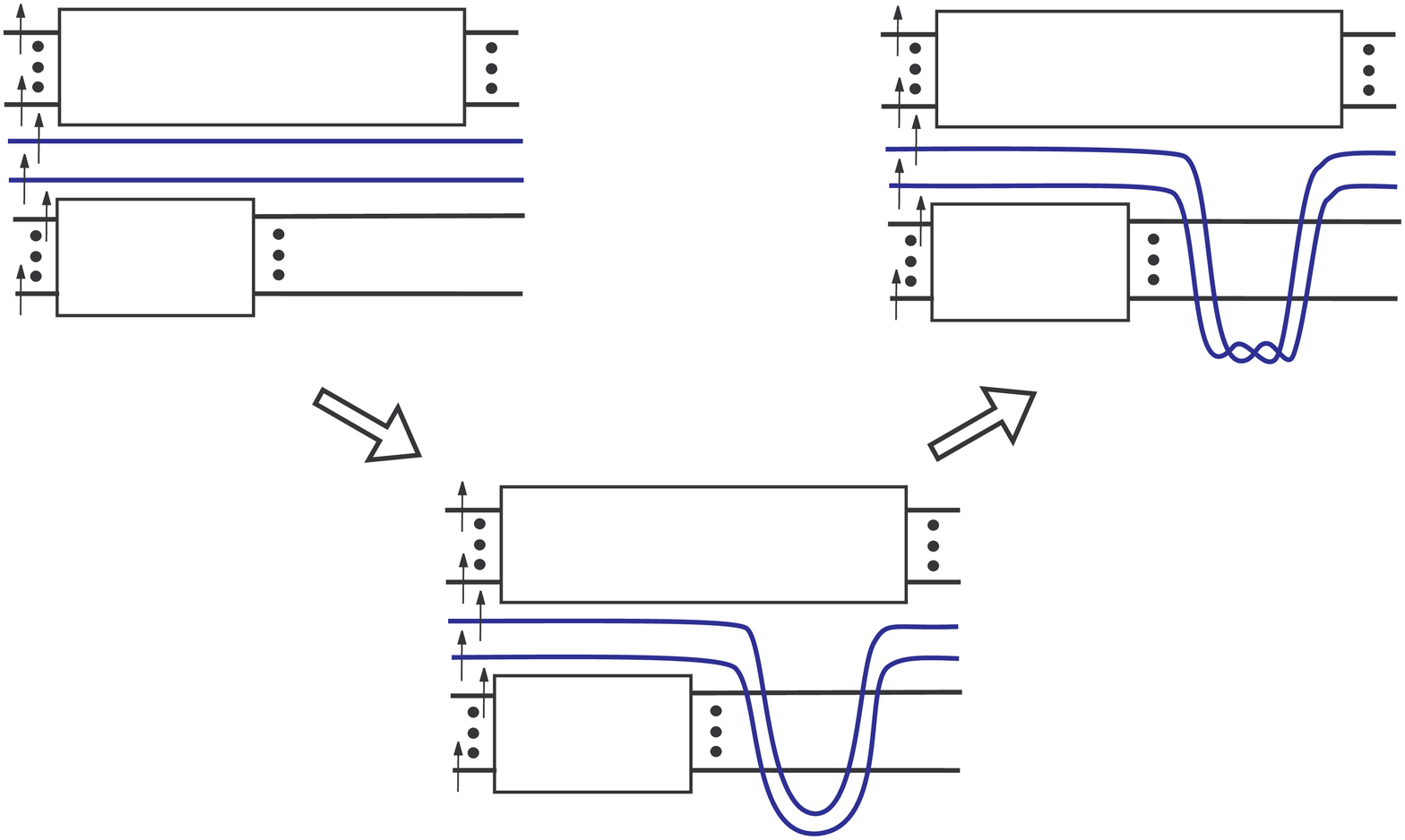}
\caption{$\mathrm{R2}^1$ moves followed by criss-cross braiding.
In the small boxes, we may have any braid (on the corresponding number of strands).}
\label{fig481}
\end{figure}

Using always-realizable base diagram moves, we will turn this pure virtual base braid on $m$ strands into a trivial braid on $m$ strands. As we will keep the normal directions on all strands, the result will also be a directed indefinite fibration, but with embedded round image.

\smallskip
\noindent \textit{\underline{Step 2:} } We first push all Lefschetz singularities away from the base diagram, into a small open disk neighborhood $N_{\text{SP}}$ of the south pole, which we will regard as the point at infinity. The only time our modifications will involve this neighborhood is when we will swing a subarc of a round image component over the south pole. In this case we will always have the normal arrow on the arc pointing towards the south pole when we begin sliding it. So we can push all the Lefschetz singularities against it to continue the slide. For these reasons, we will not discuss Lefschetz singularities any further. 

Regard the base braid diagram of our pure virtual braid as a simultaneous graph of $m$ continuous functions $[0,1] \to (0,1)$. Take the strand $b$ whose end points are in the top most position; we will refer to it as the \emph{top strand}. Each time it has a local minimum (resp.\ a local maximum) in the interior, by using $\mathrm{R2}^1$ move repeatedly, pull it down (resp.\ raise it up) ---while avoiding braid crossing--- until it becomes a global minimum (resp.\ global maximum); see Figure~\ref{EmbFig1}. Repeat this for every local minimum and maximum of the same strand, until any local minimum/maximum $b$ has is a global minimum/maximum. The number of crossings may increase drastically during this procedure!

\begin{figure}[htbp] 
\includegraphics[width=\linewidth,height=0.15\textheight,
keepaspectratio]{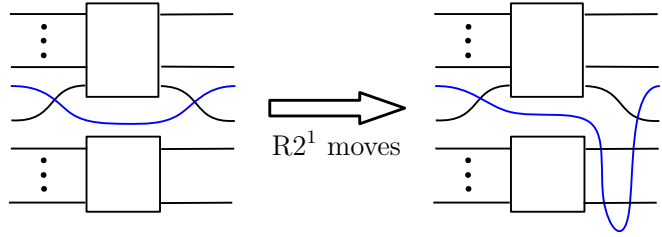} 
\caption{Lowering a local minimum of a strand (shown in blue) to a global minimum position. Normals to all strands are directed towards the top. In the two small boxes,
we may have any braids (on the corresponding numbers of strands).}
\label{EmbFig1}
\end{figure}

While the new crossing pattern between the strand $b$ we pulled around and the rest of the braid is clear, in between every consecutive global maxima (or minima) of $b$, lies some possibly highly non-trivial ---and not necessarily pure--- virtual braid on $m-1$ strands, which is split from our strand. We keep each one of these subbraids on $m-1$ strands in a box; see Figure~\ref{EmbFig2}. Moreover, we can use the braid closure to bring the boxes on the far left and far right together, and regard it as a single box. The closure of the resulting braid is shown in Figure~\ref{EmbFig2}, with normals to all strands still directed towards the north pole. 

\begin{figure}[htbp]
\psfrag{b}{$b$}
\psfrag{bb}{$b'$}
\psfrag{p}{$p$}
\psfrag{q}{$q$}
\psfrag{hb}{$\hat{b}$}
\psfrag{b1}{$B_1$}
\psfrag{b2}{$B_2$}
\psfrag{b3}{$B_3$}
\psfrag{br1}{$B_{2r-1}$}
\psfrag{br}{$B_{2r}$}
\includegraphics[width=\linewidth,height=0.3\textheight,
keepaspectratio]{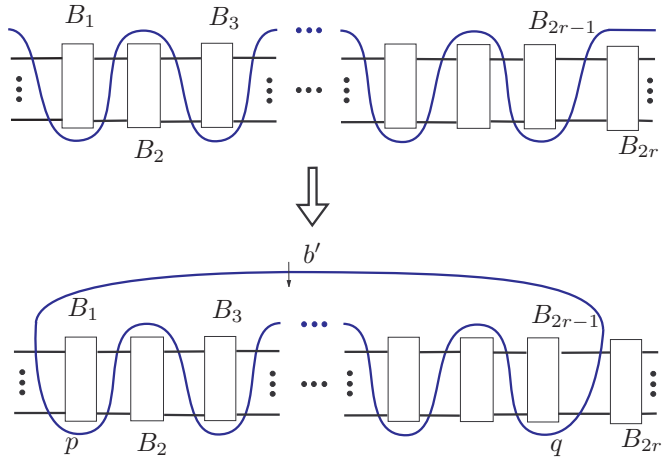} 
\caption{Strand $b$, depicted by blue, and the
other strands made up with braid boxes $B_1, B_2, \ldots, B_{2r}$ (Top).
Then, we slide the sub-strand $b'$ of $b$ over the north pole (Bottom).}
\label{EmbFig2}
\end{figure}

Observe that the number of boxes we have at the end is even; say $2r$. Here $r=0$ means 
the top strand $b$, and thus its closure $\hat{b}$, the innermost component of $Z_f$ (with respect to the north pole), do not intersect the others. In this case, we can isotope it into a small open disk neighborhood $N_{\text{NP}}$ of the north pole. Any arc we swing over the north pole with its normal arrow directed towards the north pole can go past this component by an $\mathrm{R2}^1$ move followed by an $\mathrm{R2}_2$ move. With a slight abuse, we will regard the portion of the round image $f(Z_f)$ in $S^2 \setminus N_{SP}$ as the ``base diagram''.

If $r>0$, then let $p$ and $q$ be the two global minimum points on $b$ closest 
to the ends of the strand $b$ in the diagram. (If there is only one global minimum, 
flatten the strand around it, so $p$ and $q$ are distinct points.) During the next
modifications, we will temporarily ruin the braid picture, and will only consider 
the base braid as a \emph{portion} of the round image $f(Z_f)$. Take the complement
$b'$ of $b$ in its closure $\hat{b}$. Slide $b'$ over the neighborhood $N_{\text{NP}}$ 
of the north pole so it approaches the non-trivial braid from the top; see 
Figure~\ref{EmbFig2}. The crucial point here is that $b'$ is directed \emph{against} 
all the braid components, so after several $\mathrm{R2}^0$, $\mathrm{R2}_2$, 
and $\mathrm{R3}_3$ moves (many of them involving $b$ as well), we can pull $b'$ all 
the way down, below any of the other braid strands (including minima of $b$), 
except for two kinks we unavoidably get around $p$ and $q$ on $b$. 
See the top diagram of Figure~\ref{EmbFig3}. 
During these Reidemeister moves, we fix the base braid, 
except for the far left and far right $m-1$ crossings of $b$ with other strands 
we eliminated when lowering $b$ (as we pulled $b'$ down) at the ends.

\begin{figure}[htbp]
\psfrag{p}{$p$}
\psfrag{q}{$q$}
\psfrag{bbb}{$B_{2r-1}B_{2r}B_1$}
\psfrag{bb}{$b'$}
\psfrag{b1}{$B_1$}
\psfrag{b2}{$B_2$}
\psfrag{b3}{$B_3$}
\psfrag{br1}{$B_{2r-1}$}
\psfrag{br2}{$B_{2r-2}$}
\psfrag{br3}{$B_{2r-3}$}
\psfrag{br}{$B_{2r}$}
\includegraphics[width=\linewidth,height=0.5\textheight,
keepaspectratio]{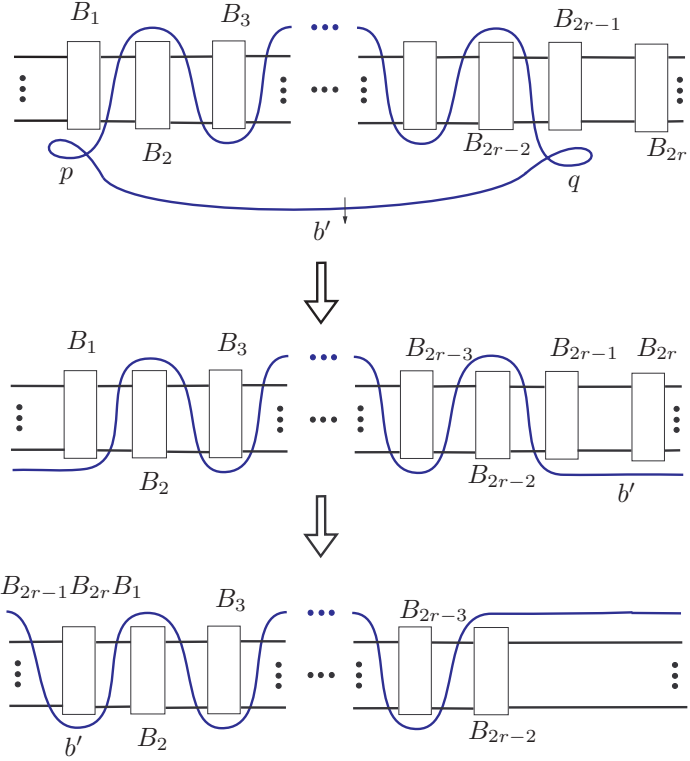} 
\caption{We lower $b'$, and then swing it over the south pole, while rearranging the
braid boxes $B_{2r-1}, B_{2r}$ and $B_1$ into one.}
\label{EmbFig3}
\end{figure}

After introducing a flip, and then applying $\mathrm{R2}_2$ and unsink moves, we can get rid of each kink as shown in Figure~\ref{fig44}. So we can continue pulling $b'$ and swing it over the south pole in order to bring it again to the other side of the base diagram. It now completes the missing strand of the trivial braid portion.

\begin{figure}[htbp]
\centering
\psfrag{f}{flip}
\psfrag{r}{$\mathrm{R2}_2$}
\psfrag{s}{unsink}
\includegraphics[width=0.8\linewidth,height=0.1\textheight,
keepaspectratio]{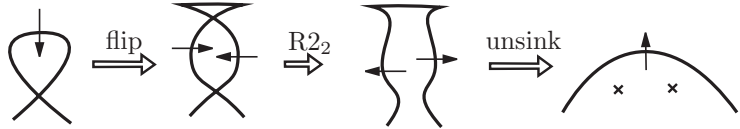}
\caption{Getting rid of the kinks.}
\label{fig44}
\end{figure}

So we have a new base braid on $m$ strands, which represents the new round image $f(Z_f)$.
The bottom strand came from sliding $b'$, 
whereas the rest of the braid diagram is the same as before (see the middle
diagram of Figure~\ref{EmbFig3}). As we get ready to repeat the whole 
procedure, we first observe that the bottom strand already has only global extrema. 
More importantly, we now have $2$ boxes on one side and $1$ box on the other side 
which do not have the bottom strand going in between. 
Using the braid closure as before, we can slide them into a one box.
We now have $2(r-1)$ boxes in total 
(see the bottom diagram of Figure~\ref{EmbFig3}). 

If the new number of boxes is not zero, then it means we still have some 
global maximum. Then, isotoping the round image, we can place the same component of 
$f(Z_f)$ in the braid diagram so that it is in the top position again. Hence, 
repeating the procedure $r-1$ more times, we arrive at a base braid with 
a bottom strand split from the others. 

For one last time, we slide the complement $b'$ of the new top strand $b$ over the north pole and across the braid diagram. Though this time we bring it below the rest of the braid diagram but keep it above $b$ (creating no kinks), so the round image component $b \cup b'$ is embedded and encloses a disk region without any singularities. Its normal arrow is directed towards the interior of the disk. We can isotope this component of the new round image $f(Z_f)$ into the small neighborhood $N_{\text{SP}}$ of the south pole, next to the Lefschetz singularities. As it was the case for a component we isotoped into $N_{\text{NP}}$, any arc we swing over $N_{\text{SP}}$ with its normal arrow directed towards the south pole can go past this component by an $\mathrm{R2}^1$ move followed by an $\mathrm{R2}_2$ move. 

\smallskip
\noindent \textit{\underline{Step 3:} }  With our remarks on the singular image isotoped into $N_{\text{NP}} \cup N_{\text{SP}}$ in mind, we can run \emph{Step 2} for the remaining virtual braid on $m-1$ strands. Repeating it $m-2$ times, we obtain a new indefinite fibration $f$ with embedded singular image. The singular image of $f$ is now contained in the two open disks $N_{\text{NP}}$ and $N_{\text{SP}}$, with $f(Z_f)$ consisting of $m$ split components (i.e.\ no two are nested). 
We view all in a large disk $D$ as shown on the left hand side of Figure~\ref{split2directed} below. Recall that the fiber over the south pole was connected. So the fiber over the south pole after Steps 1--2 is still connected: as we swung several indefinite fold arcs over this point, we only added $1$--handles to it. It follows that the regular fibers over the boundary of the disk $D$ are connected.

\begin{figure}[htbp] 
\includegraphics[width=\linewidth,height=0.27\textheight,
keepaspectratio]{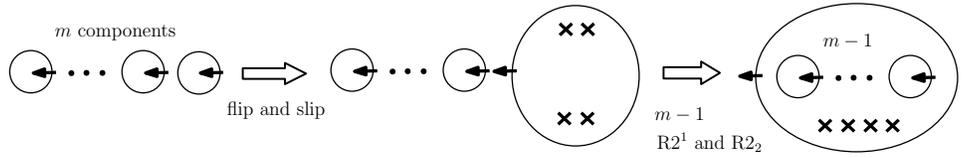} 
\caption{From split round image back to directed.}
\label{split2directed}
\end{figure}

\smallskip 
\noindent \textit{\underline{Step 4:} } Here we apply the \emph{flip and slip move} of 
\cite[Fig.~5]{B2}, to turn an inward-directed circle in $D$ inside out, so it becomes outward-directed (see Figure~\ref{fig151} in Subsection~\ref{ss:mixed}).

Apply flip and slip to the far right circle so it is now outward-directed. By exactly $m-1$ $\mathrm{R2}^1$ and $m-1$ $\mathrm{R2}_2$ moves again, we can pull the left half of the circle all the way to the left, so it now encloses all the other circles; see Figure~\ref{split2directed}. Repeating this procedure $m-2$ times for the split collection of circles contained inside, we arrive at a nested collection of outward-directed, embedded round image circles. All the Lefschetz singularities we had and produced during this process can be pushed into the innermost round image circle.

This is the base diagram for the resulting directed indefinite fibration 
$g\colon X \to S^2$ with embedded round image. Because the regular fibers over 
$\partial D$ are connected, and because all the others are obtained 
from them by fiberwise \mbox{$1$--handle} attachments and Lefschetz handle 
attachments, the  indefinite fibration $g$ is fiber-connected. 
\end{proof}

\smallskip
\begin{rk}  \label{rk:otherdim}
Theorems~\ref{mainthm1} and~\ref{mainthm2} together show that given any map from a closed orientable $4$--manifold to $S^2$, we can homotope it  to a generic map with embedded singular image. The corresponding statement is not true for an arbitrary map from a $3$--manifold to $S^2$ (or to other surfaces); there are obstructions to getting a generic map with embedded singular image within the same homotopy class \cite{CT,Gromov, Sa3}. Similarly, there are obstructions for maps from $4$--manifolds to $3$--manifolds \cite{SY}. \linebreak In this regard, maps from $n \geq 4$ dimensional manifolds to $S^2$ seem rather special. We hope to explore this phenomenon for dimensions $n >4$ in future work.
\end{rk}


\smallskip
\subsection{Connected fibers and connected round locus} \

Here we will prove the following proposition, which, together with Theorems~\ref{mainthm1} and~\ref{mainthm2}, provides a sequence of base diagram moves to homotope any given indefinite fibration to one that is fiber-connected and has connected round locus --in short, \emph{simplified}. 

\begin{prop}\label{prop:locusconnected}
Let $X$ be a connected $4$--manifold. There exists a finite algorithm consisting of sequences  of flip, unsink, push,  cusp merge, and Reidemeister type moves $\mathrm{R2}^1$ and $\mathrm{R2}_2$, which homotopes any given inward-directed indefinite fibration $f\colon X \to D^2$ with embedded round image to a fiber-connected outward-directed indefinite fibration \mbox{$g\colon X \to D^2$} with embedded round image and connected round locus.
\end{prop}

\begin{figure}[htbp] 
\includegraphics[width=\linewidth,height=0.27\textheight,
keepaspectratio]{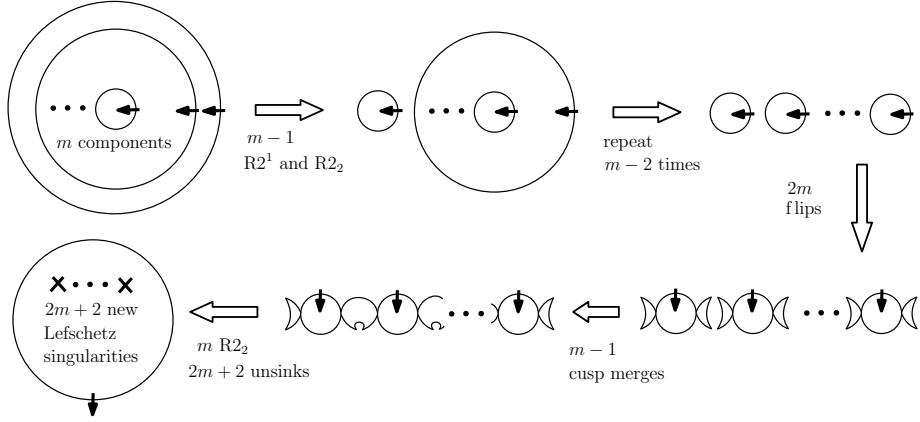} 
\caption{Making fibers and the round locus connected.}
\label{blackboard}
\end{figure}

\begin{proof}
Observe that since $X$ is connected, even if there are disconnected 
fibers, the regular fibers over $\partial D^2$ should be connected.

The first part of the algorithm is the same as in Steps 2--3 of our proof 
of Theorem~\ref{mainthm2}, carried out in much simpler case of a pure virtual 
braid. We first push out all the Lefschetz singularities so they are near the 
boundary of the disk, and do not interfere with all the other moves we will perform. 
Assume that there are $m >1$ components. By exactly $m-1$ $\mathrm{R2}^1$ and 
$m-1$ $\mathrm{R2}_2$ moves, we can pull out the 
right half of the outermost circle all the way left, so it is now disjoint 
from the rest. Repeating this $m-2$ times for the next outermost circle of 
the nested circles each time, we arrive at a split collection of $m$ inward-directed embedded circles as in the top row of Figure~\ref{blackboard}. 

Now flip each circle twice, and then merge all into one immersed circle 
by $m-1$ cusp merges as shown in Figure~\ref{blackboard}. 
(Note that these cusp merges are realizable, since the fibers over points
near the boundary are all connected.)
This is followed by $m$ $\mathrm{R2}_2$ and $2m+2$ unsink moves to arrive at an outward-directed embedded round image, where the resulting indefinite fibration $g\colon X \to D^2$  has connected round locus. Once again, it is fiber-connected, 
since all the fibers are obtained by fiberwise \mbox{$1$--handle} attachments and Lefschetz handle attachments.
\end{proof}

\smallskip
\begin{rk}\label{remark:boundary}
The procedures we have given in Theorem~\ref{mainthm1}, in Step 1 of Theorem~\ref{mainthm2} and in Proposition~\ref{prop:locusconnected} can all be applied in a more general setting. 
Let $X$ be a compact oriented $4$--manifold with corners, $\Sigma$ be a compact oriented surface with boundary, 
and $f \colon X \to \Sigma$ be
an indefinite fibration such that 
\begin{itemize}
\item $\partial X = P \cup Q$, where $P$ and $Q$
are compact codimension zero submanifolds of $\partial X$
and are glued along $\partial P = \partial Q = P \cap Q$,
\item $X$ has corners exactly along $\partial P = \partial Q$,
\item $f^{-1}(\partial \Sigma) = Q$,
\item $f|_Q \colon Q \to \partial \Sigma$ and $f|_P \colon P \to \Sigma$ are submersions.
\end{itemize}
In particular, when $\Sigma = D^2$, we have a naturally induced open book structure on $\partial X$.
The moves we have described can be seen to work for such an indefinite fibration as well. For birth/death, flip/unflip,
fold merge, criss-cross braiding, and Reidemeister type moves, the handlebody arguments we used are implicitly 
based on ascending and descending
manifolds for gradient-like vector fields (for example, see \cite{GK2}).
Let us take an embedded arc $\alpha$ in $\Sigma$ which is
transverse to $f$. Then, $f|_{f^{-1}(\alpha)} \colon
f^{-1}(\alpha) \to \alpha$ is a Morse function. As $f$ is a submersion
on $P$, the fibers of $f$ are compact surfaces with boundary
and this Morse function is a submersion on the closure of
the boundary of $f^{-1}(\Int{\alpha})$. So we can
choose a gradient-like vector field that is tangent to the
boundary. This means that an integral curve never hits the
boundary (except along $f^{-1}(\partial \alpha)$) as long as 
we start from an interior point. Therefore, the above mentioned moves can be realized by
exactly the same method as before. Unsink, push, cusp merge, wrinkle, and $C$-moves are easily seen to be realizable as well.
\end{rk}

\smallskip
\section{Realizing a prescribed $1$--manifold as the round locus} \label{Sec:Realization}

The purpose of this section is to show that any null-homologous $1$--dimensional embedded 
closed oriented submanifold $Z$ of $X$, with any prescribed twisting data 
satisfying a certain necessary condition 
(for the local models around components) can be 
realized as the round locus $Z_f$ of a fiber-connected, directed broken Lefschetz 
fibration $f\colon X \to S^2$ with embedded 
round image.

First, we note a couple observations on the round locus of a broken Lefschetz fibration/pencil, which will naturally 
appear as necessary conditions for our arguments to follow. Analogous statements for the zero loci of near-symplectic structures 
are well-known \cite{P2}, so one can also appeal to \cite{ADK} to build a near-symplectic form whose zero locus is the same 
as the round locus, and translate these results. We will instead sketch parallel arguments in the realm of indefinite fibrations. 

\begin{prop}\label{Nullhomologous} Let $f: X \to \Sigma$ be an indefinite fibration, where
$X$ and $\Sigma$ are closed. Then,
$Z_f$ is a closed $1$--dimensional submanifold of $X$ which is canonically oriented. Furthermore, 
$Z_f$ is null-homologous, i.e., $[Z_f]=0$ in $H_1(X; \Z)$.
\end{prop} 

\begin{proof}
If there are Lefschetz singular points, then use a wrinkling move to perturb each Lefschetz singular point into an indefinite 
singular circle with $3$ indefinite cusps. Born in a $4$--ball, each one of these circles is null-homologous. 
Thus we may as well assume that $f$ is an indefinite generic map. Then by \cite{An}, the singular point set $Z_f$ of $f$ 
is orientable. Moreover, it is shown in \cite{An} that the Poincar\'e dual of  $[Z_f] \in H_1(X ; \Z)$ is equal to the 
integral Stiefel--Whitney class $w_3(X) \in H^3(X; \Z)$. 
However, $w_3$ is the obstruction to admitting a spin$^c$ structure, and it vanishes on a closed oriented 
$4$--manifold $X$; see \cite[\S4]{HH}, \cite{Wh0}. 
\end{proof}

\smallskip
\begin{prop} \label{untwisted}
Let $f \colon X \to \Sigma$ be a broken Lefschetz fibration, where
$X$ and $\Sigma$ are closed and connected. 
Then, the number of untwisted (even) components 
of $Z_f$ has the same parity as $1 + b_1(X) + b_2^+(X)$. 
\end{prop}
\vspace{-0.2cm}
\begin{proof}
Here is a proof verbatim to the one given in \cite{P2} for near-symplectic forms. Given a Riemannian metric on $X$, 
we can naturally construct an almost complex structure on $X \setminus Z_f$. Off the Lefschetz critical points, we define the 
almost complex structure as the rotation of $\pi/2$ on the tangent planes to the regular part of fibers and also on the normal 
planes. Around Lefschetz critical points, this is modified to match the complex structures coming from the local models. 
Then, the obstruction to extending the above almost complex structure to whole of $X$ coincides with both the 
number of untwisted circles in $Z_f$
and $1 - b_1(X) + b_2^+(X)$ modulo $2$ \cite{P2}.
\end{proof}

We can now prove the main result of this section:
\begin{thm} \label{dei}
Let $X$ be a closed connected oriented $4$--manifold and $Z$ be a 
(non-empty) null-homologous 
closed oriented $1$--dimensional submanifold of $X$. There exists a finite algorithm 
consisting of sequences of always-realizable moves flip, cusp merge, unsink, push, 
criss-cross braiding, and multi-germ moves $\mathrm{R2}^0$, $\mathrm{R2}^1$, $\mathrm{R2}_2$, 
$\mathrm{R3}_2$ and $\mathrm{R3}_3$, which homotopes any given indefinite fibration 
$f \colon X \to S^2$ with non-empty round locus to a fiber-connected, directed broken 
Lefschetz fibration \mbox{$g\colon X \to S^2$} with embedded round image, whose round 
locus $Z_g$ coincides with $Z$ as oriented \mbox{$1$--manifolds}. 
Moreover, for any prescribed twisting data for $Z$, with number of untwisted components 
congruent \mbox{modulo $2$} to $1 + b_1(X) + b^+_2(X)$, $g$ can be chosen 
so that $Z_g$ and $Z$ have matching twisting data.
\end{thm} 

\begin{proof} 
Take the indefinite fibration $f$ on $X$, and apply the procedures we presented in the proofs of 
Theorems~\ref{mainthm1} and \ref{mainthm2} to get a \emph{directed} indefinite fibration with \emph{embedded} 
round image. We can also make it fiber-connected and with connected round locus by applying the procedure given 
for Proposition~\ref{prop:locusconnected}. Note that since the base is $S^2$, each time we 
have an inward-directed or an outward-directed image, we also have the other. At the end, we have a fiber-connected 
indefinite fibration with embedded and connected round image, which we will continue to denote by $f$. 
We can also assume that $f$ has no cusps by applying unsink moves. In the following, we will modify $f$ step by step. 
For simplicity, after each step, we continue to denote the resulting map by the same symbol $f$.

Let $F$ denote the highest genus regular fiber. Then, the inclusion $i\colon F \to X$ induces an epimorphism 
$i_*\colon \pi_1(F) \to \pi_1(X)$, as seen from the handlebody decomposition of $X$ induced by the broken 
Lefschetz fibration $f$ \cite{B1}. (The kernel is normally generated by the vanishing cycles of the round 
$2$--handle and Lefschetz $2$--handles, together with the attaching circle of one more $2$--handle pulled 
from the lower side.)

\begin{figure}[htbp]
\centering
\psfrag{f}{flip}
\psfrag{m}{cusp merge}
\psfrag{r}{$\mathrm{R2}_2$}
\psfrag{s}{unsink}
\psfrag{i}{isotopy}
\psfrag{i2}{isotopy on $S^2$}
\psfrag{o}{on $S^2$}
\psfrag{t}{P}
\includegraphics[width=\linewidth,height=0.7\textheight,
keepaspectratio]{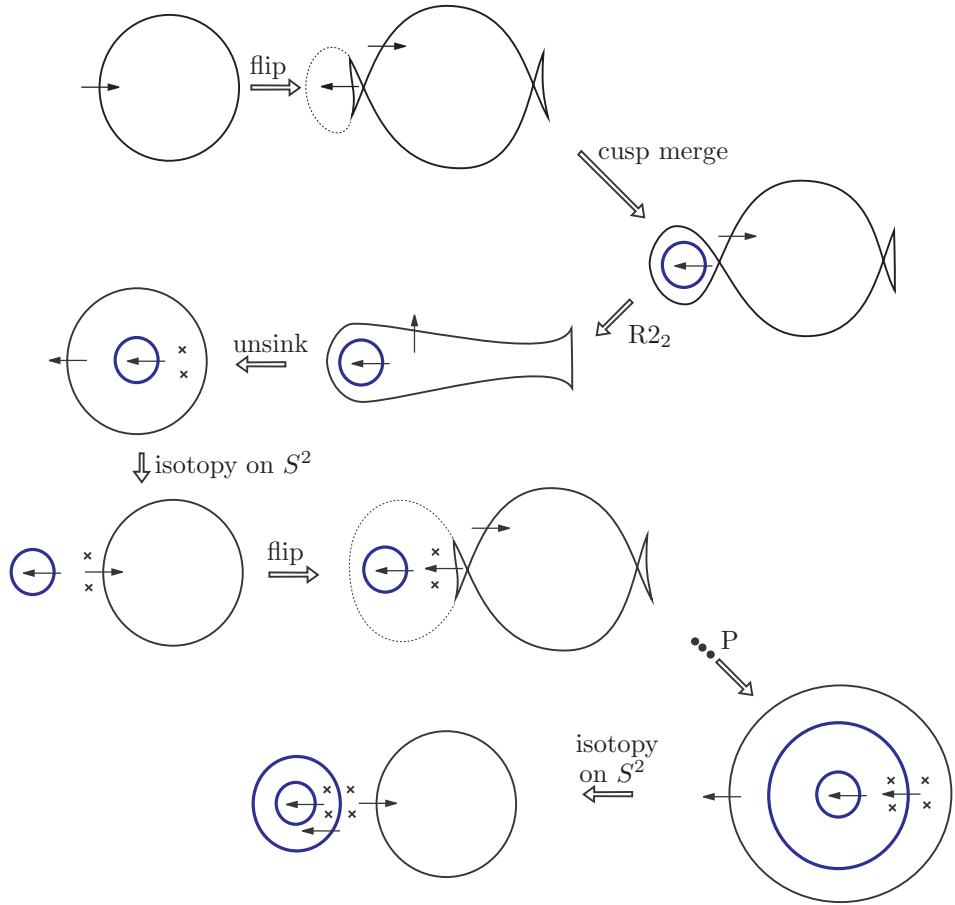} 
\caption{Realizing prescribed components, whose images are depicted by 
blue circles.}
\label{fig181}
\end{figure}

Now, by applying the moves as in Figure~\ref{fig181}, we can make prescribed components in $Z_f$ 
in such a way that all but one component of $Z_f$ are isotopic to the corresponding components 
of $Z$. These are mainly base diagram moves, except we pay additional attention to how we perform 
cusp merges: we merge them along paths which approximate the components of $Z$ we would like to realize. 
Also, we first assume that the component of $Z$ in question is isotoped so that the restriction of 
$f$ to it is already an embedding, the image of which is given by a blue 
circle in the figure. 
(This is achieved by representing the homotopy class of the component by
a loop $\gamma$ in a highest genus fiber $F$ and then by perturbing it to the
loop defined by $S^1 \ni t \mapsto (\gamma(t), t) \in F \times S^1$, where
$S^1$ is a small embedded circle in the base surface and $F \times S^1 = f^{-1}(S^1)$.)
Note that the path for each cusp merge move, depicted by a dotted line, lies in a region corresponding to the highest genus region before the preceding flip. Therefore, we can adjust the homotopy classes (and hence the isotopy classes) of the components of $Z_f$ corresponding to the blue circles, by the surjectivity of $i_*$ mentioned above. The process indicated by the letter ``P'' in Figure~\ref{fig181} consists of a repetition of the preceding operations.

\begin{figure}[htbp]
\centering
\psfrag{f}{flip}
\psfrag{m1}{cusp}
\psfrag{m2}{merge}
\psfrag{r}{$\mathrm{R2}_2$}
\psfrag{e}{$\tilde{R}$}
\psfrag{i}{isotopy on $S^2$}
\psfrag{a}{$\alpha$}
\psfrag{b}{$\beta$}
\includegraphics[width=\linewidth,height=\textheight,
keepaspectratio]{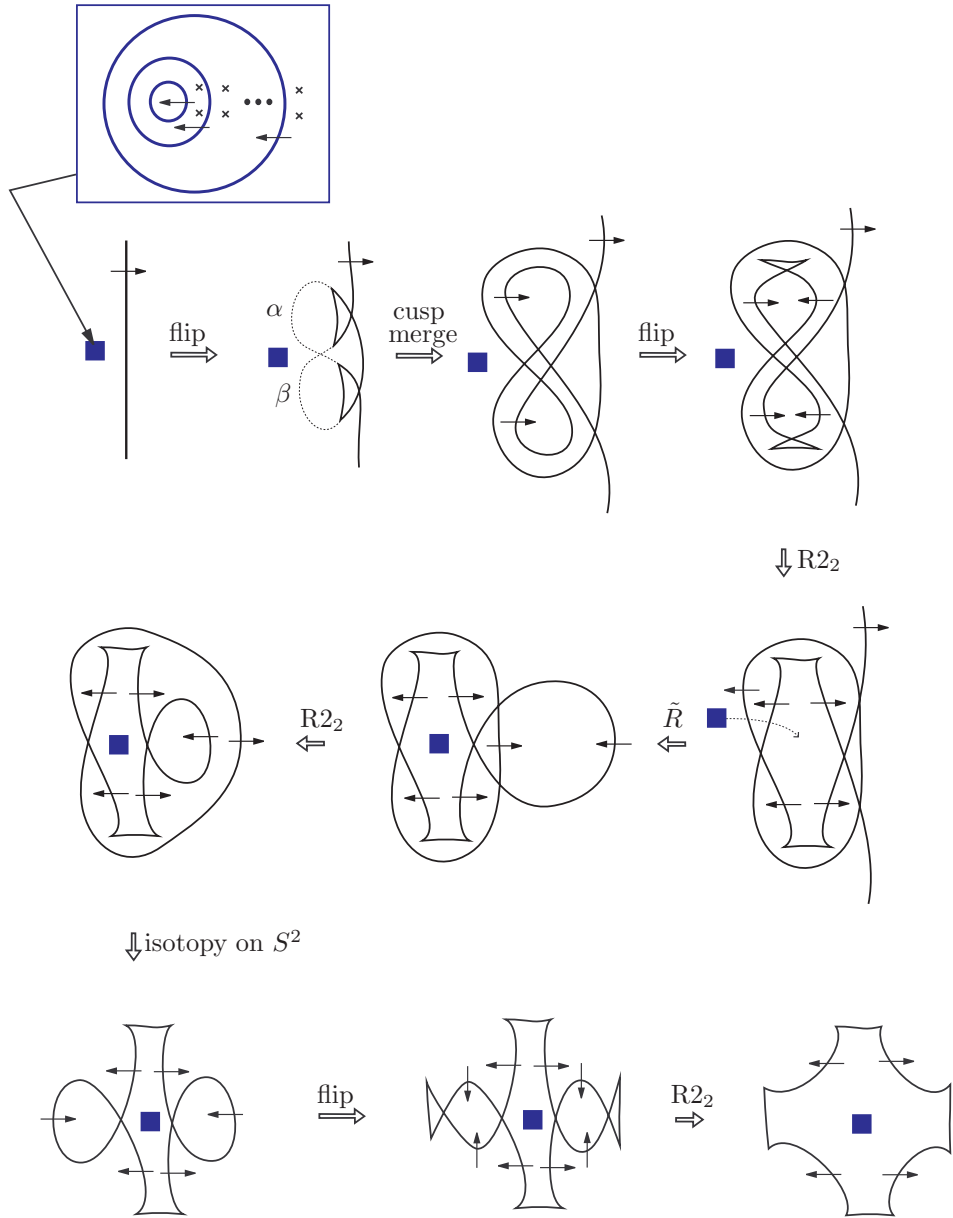}
\caption{Adjusting the last prescribed component.}
\label{fig191}
\end{figure} 

By assumption, $Z$ is null-homologous, and so is $Z_f$ by Proposition~\ref{Nullhomologous}.
Since we already matched all but one components, the remaining component of $Z_f$ is \mbox{$\Z$--homologous} to  the remaining component of $Z$. Then, their difference as elements of the fundamental group $\pi_1(X)$ 
is a product of finitely many commutators, again surjected from $\pi_1(F)$. By applying the base diagram 
moves as in Figure~\ref{fig191},
we can adjust the last component of $Z_f$ by one commutator at a time, say
$[\alpha, \beta] = \alpha \beta \alpha^{-1} \beta^{-1}$. Note that in the blue 
box in Figure~\ref{fig191}, we have 
concentric inward-directed circles together with Lefschetz critical values. Here $\tilde{R}$ in Figure~\ref{fig191} consists of an 
iteration of pushes, moves $\mathrm{R2}^1$ and moves $\mathrm{R2}_2$, which are always-realizable. 
At the final stage, we apply the unsinks eight times to get the image \mbox{$f(Z_f) \cup f(C_f)$} exactly the same as 
the original one except for the Lefschetz critical values. However, the position of $Z_f$ in $X$ has been changed in such 
a way that the homotopy class of the final component has been changed by a commutator.

By repeating this procedure, we arrive at a broken Lefschetz fibration $g\colon X \to S^2$ with directed embedded 
image and with $Z_g$ = $Z$, which is homotopic to the original $f$ through always-realizable moves listed in 
the statement of the theorem. This proves the first part of our claim.

As for matching the twisting data, we need to slightly modify the above proof. Note that the condition on the number of untwisted components is necessary by  Proposition~\ref{untwisted}. 

In order to adjust the homotopy class of a new born component, we used a cusp merge along a joining curve in the highest 
genus region. Given such a curve, we have some freedom for performing the cusp merge \cite{L, BeH}. As seen in the 
proof of (4.8) Lemma~(1) or (4.6) Lemma~(2) in \cite{Lev}, an appropriate set of coordinates is chosen in 
a neighborhood of the joining curve. The choice is, in a sense, canonical, except that we have some freedom to 
choose coordinates along fibers. More precisely, we can change the coordinates $z$ according to the parameter $u$.
If we choose the coordinates $z$ in such a way that it is ``rotated by $\pi$'' along the parameter $u$, then the result of the cusp merge looks the same; however, the local data changes along the two new born arcs. (In the terminology of proof of \cite[Theorem~6.1]{L} or in the argument
in the last part of merging move in Section~3 of \cite{L}, we can rotate the $2$--disks embedded in the fibers along the horizontal curve connecting the two Lefschetz critical points.) Using this technique, we can adjust the twisting data of the new born component. We continue the same process until we get the final component. The twisting data of this last component \emph{will be} the same as the twisting data of the corresponding component of $Z$, since the number of untwisted components of both have the same parity, namely, $1 + b_1(X) + b^+_2(X)$ modulo $2$.
\end{proof}

\smallskip
\begin{rk}
We can remove the condition $Z_f \neq \emptyset$ from Theorem~\ref{dei} by allowing an initial birth move; otherwise, our procedure features no births or death moves. A similar realization is also possible for generic maps into surfaces; see \cite{Sa0}. 
\end{rk}

\smallskip
\section{Constructions of broken Lefschetz fibrations and pencils} \label{Sec:Assembly} 

We now present purely topological constructions of broken Lefschetz fibrations and pencils with simplified topologies. These algorithmic constructions improve on the procedures given by the authors in \cite{Sa, B2}, by incorporating the new procedures we have obtained in the previous sections.

\smallskip
First is the construction of broken Lefschetz fibrations on arbitrary $4$--manifolds:

\begin{thm} \label{thmA} 
Let $X$ be a closed connected oriented $4$--manifold and $Z$ be a (non-empty)
null-homologous closed oriented $1$--dimensional submanifold of $X$, given with a prescribed twisting 
data, in which the number of untwisted (even) components has the same parity as $1 + b_1(X) + b^+_2(X)$. 
Then, there exists a fiber-connected, directed broken Lefschetz fibration 
\mbox{$f\colon X \to S^2$} with embedded round image, whose round locus $Z_f$ matches $Z$ 
with the same local models. 
Given any generic map from $X$ to $S^2$, such $f$ can be derived from it by an explicit algorithm.
\end{thm}

\begin{proof}
Let $h\colon X \to S^2$ be a generic map, which always exists \cite{Lev0, T, Wh0}. 
It has only fold and cusp singularities, but
the $1$--dimensional singular locus might include definite folds. If that is the case, 
then we homotope $h$ to an indefinite 
generic map \mbox{$f\colon X \to S^2$} using an algorithm given by the second author in \cite{Sa} (also see \cite{Sa2}).
This procedure is given by moves similar to the ones we have discussed here, but now they
involve \emph{definite} folds as well: always-realizable flip, cusp--fold crossing,
birth--death, and Reidemeister type moves, together with cusp merge and fold merge moves
that can be realized algorithmically.\footnote{In \cite{Sa}, surgery moves were used algorithmically, while in \cite{Sa2}, 
another definite fold elimination technique is introduced without involving such surgery moves.}
An alternate proof for eliminating the 
definite fold, which also goes through a sequence of modifications of a generic map by homotopy moves, is given in 
\cite{GK2}.

We can now apply our algorithm for Theorem~\ref{dei} to homotope this indefinite fibration to 
a directed broken Lefschetz 
fibration with connected fibers and embedded round image, whose round locus realizes $Z$ with its prescribed 
twisting data (for the local models). This procedure already includes the algorithms for Theorems~\ref{mainthm1} 
and \ref{mainthm2} to obtain a directed indefinite fibration with embedded round image, and the ones for 
Proposition~\ref{prop:locusconnected} to make all the fibers connected. 
All is achieved by always-realizable moves flip, cusp merge, unsink, push, criss-cross braiding, and the multi-germ 
moves $\mathrm{R2}^0$, $\mathrm{R2}^1$, $\mathrm{R2}_2$, $\mathrm{R3}_2$ and $\mathrm{R3}_3$.
\end{proof}

As a corollary, we obtain a purely topological and algorithmic construction ---from any given generic map--- of \emph{simplified broken Lefschetz fibrations}  introduced in \cite{B1}, as well as that of \emph{simplified wrinkled fibrations} introduced in \cite{W} on arbitrary  closed $4$--manifolds. (These are simplified indefinite fibrations without cusps and without Lefschetz singularities, respectively.)

\begin{cor} \label{simpBLF}
There is an explicit algorithm, consisting of always-realizable base diagram moves, which turns any generic map 
from a closed oriented $4$--manifold $X$ to $S^2$ to a simplified broken Lefschetz fibration 
(or to a simplified wrinkled fibration). Therefore, any closed oriented connected $4$--manifold $X$  admits a simplified broken Lefschetz fibration.
\end{cor}

\begin{proof}
The algorithm for a simplified broken Lefschetz fibration on $X$ is provided by Theorem~\ref{thmA} by taking $Z$ as 
\emph{any} null-homologous circle in $X$, e.g.\ a null-homotopic one. Note that the twisting type of $Z$, since it has 
only one component, is already governed by the topology of $X$ by Proposition~\ref{untwisted}.  We can also turn a 
simplified broken Lefschetz fibration to a simplified wrinkled fibration \cite{W}: perform flip-and-slips on the round image, 
but without unsinking the cusps at the end, perturb each Lefschetz singularity using the wrinkling move, and cusp merge 
all the components (using arcs that project to embedded ones under the fibration map) to one. 
\end{proof}

\begin{rk}
The ``simplified'' term for an indefinite fibration coins all kinds of simplifications we have considered: directed, embedded 
round image, fiber-connected, connected round image (and all Lefschetz singularities on the higher side if it is a broken 
Lefschetz fibration).  A simplified indefinite fibration yields a beautifully simple presentation of the $4$--manifold in terms of loops 
on a surface \cite{B1, W}. These presentations are far from being unique, though, and the reader should be careful 
when implementing our algorithms. There are choices in many of the always-realizable base diagram moves we employed, such as 
the cusp merge or the Reidemeister II type moves, which will result in different presentations on the same surface. \mbox{(See, e.g., \cite{H3, BeH}.)}
\end{rk}

\begin{rk} 
As we discussed in the Introduction, these are the first purely topological 
and explicit constructions of broken Lefschetz fibrations on arbitrary $4$--manifolds with embedded round images. 
Earlier handlebody proofs \cite{GK1, L, B2, AK}, which start with an arbitrary 
Morse function 
(similar to our initial generic function), would break the $4$--manifold $X$ into 
two $2$--handlebodies $X_i$, $i = 1, 2$,
equip them with broken Lefschetz fibrations with open book boundaries, and match 
the latter implicitly by 
invoking powerful results of Eliashberg and Giroux from contact 
geometry \cite{Eliashberg, Giroux}. 
One exception is the particular case of \emph{doubles}, i.e.\ when 
$X_2=-X_1$, where the open books readily match 
\cite{GK1, Hughes}. These however constitute a rather small class 
of \mbox{$4$--manifolds.}
\end{rk}

\smallskip

Next, we provide a construction of broken Lefschetz pencils on near-symplectic \mbox{$4$--
manifolds.}

\begin{thm} \label{thmB} 
Let $\omega$ be a near-symplectic form on a closed oriented  \mbox{$4$--manifold} $X$ with non-empty $Z_{\omega}$. 
Then, there exists a fiber-connected directed broken Lefschetz pencil $f$ on $X$ with embedded round image, 
whose round locus $Z_f$ coincides with $Z_{\omega}$ with the same twisting data, and so that $\omega([F])  > 0$ 
for any fiber $F$ of $f$.  Given any generic map from $X \setminus B$ to $S^2$, which has a regular fiber Poincar\'{e} 
dual to an integral near-symplectic form $\omega$, and is a projectivization (i.e.\ conforming to the local model 
$(z_1, z_2) \mapsto z_1/z_2$) around each point in a discrete set $B$ of cardinality $[\omega]^2$  in $X$, 
such $f$ can be derived from it by an explicit algorithm.
\end{thm}

\noindent Generic maps, which satisfy the conditions listed in the last sentence so as to be prototypes for pencils, 
are found in abundance, as we will demonstrate below.

\begin{proof} 
We can assume that $\omega$ is integral: if needed, approximate $\omega$ by a rational near-symplectic form with 
the same zero locus (and twisting data), and take a positive multiple of it. Let $F$ be a closed 
oriented surface representing its Poincar\'e dual. Since $\omega^2>0$, we have $[F]^2 = m > 0$, 
for some integer $m$. Let $\tilde{X}$ be the blow-up of $X$ 
at $m$ points on $F$ and in the complement of $Z_\omega$, \,$E_1, E_2, \ldots, E_m$\, the exceptional 
spheres, and $\tilde{F}$ the 
proper transform of $F$.  
So $[\tilde{F}]^2=0$ and each $E_j$  intersects $\tilde{F}$  positively and transversely at one point.

Let $N_j$ be disjoint tubular neighborhoods of $E_j$, for $j=1, 2, \ldots, m$, and $N_0$ be 
a tubular neighborhood of 
$\tilde{F}$ in $\tilde{X}$. Set $h_0 \colon N_0 \to D^2$ to be the projection onto 
the $D^2$--factor of the trivialization 
$N_0 \cong D^2 \x \tilde{F}$, where the target $D^2$ is embedded as the southern
hemisphere of $S^2$,
and $h_j : N_j \to S^2$ to be the radial projection onto $E_j \cong S^2$. We arrange the latter so that all the 
$D^2$--sections of $E_j \cap N_0$ are mapped onto the southern hemisphere of $S^2$
and that $h_j$ coincides with $h_0$ on $N_0 \cap N_j$, $1 \leq j \leq m$. We can now define a surjective 
map $h_N$ from  $N = \bigcup_{j=0}^m N_j$ onto $S^2$. Here, the preimages of the points of the southern 
hemisphere are diffeomorphic to $\tilde{F}$, and the preimages of the interior points of the 
northern hemisphere consist of 
$m$ copies of $2$--disks.

Extend $h_N$ to a continuous map $h\colon \tilde{X} \to S^2$, e.g.\ by first defining the map on a collar 
of $\partial (X \setminus \Int{N})$ using $h|_{\partial N}$, and then mapping all remaining points 
in the interior to the north 
pole of $S^2$. Then approximate this $h$ by a generic map relative to $N$. Note that $h$ was 
already a submersion on $\Int{N}$.

Given such a generic map $h$, we apply the same algorithm in the proof of Theorem~\ref{thmA} to obtain a 
fiber-connected, directed broken Lefschetz fibration \mbox{$\tilde{f}\colon \tilde{X} \to S^2$} with embedded round 
image, whose round locus $Z_{\tilde{f}}$ realizes the given $Z=Z_{\omega}$ with prescribed local models. 
Here, $Z_\omega$ is identified with its image under the natural inclusion of $X$ minus 
the blown-up points into $\tilde{X}$, with the same local data.
Performing all the modifications away from $N'=\bigcup_{j=1}^m N_j$, we can guarantee that 
$\tilde{f}|_{N'}=h|_{N'}$. 
(This is possible, since $\tilde{f}$ is a submersion on $\partial (X \setminus \Int{N'})$; see 
Remark~\ref{remark:boundary}. For the elimination process of definite folds, the procedures given in \cite{B2, Sa}
work the same in the case of a manifold with boundary.)
Moreover, we can still assume that $\tilde{f}$ is a submersion over the southern hemisphere, 
but with fibers that possibly 
have different genera than the original $\tilde{F}$. (Because some of our procedures might 
use base diagram moves that 
swing a fold arc over this region.) Every fiber of $\tilde{f}$, which are all homologous to the 
original $\tilde{F}$ even if $\tilde{F}$ is no longer a fiber,
intersects each $E_j$ positively and transversely at one point. So, each exceptional sphere $E_j$ is a 
section of this broken Lefschetz fibration, and blowing-down all we obtain the desired pencil. 
\end{proof}

This provides an alternate, purely topological and constructive proof of the harder direction 
of the main result, Theorem~1, of \cite{ADK}, together with some of its enforcements, such as having a pencil with directed and embedded round image. The authors' proof in \cite{ADK} instead used approximately holomorphic techniques of Donaldson to establish the existence of these broken Lefschetz pencils implicitly.

We fall short of capturing another enforcement in \cite{ADK} that seems out of reach for
``softer'' techniques: making the fibers symplectic with respect to the given near-symplectic form 
away from its zero locus. Nevertheless, the converse result of \cite{ADK}, which is a  
Thurston--Gompf type construction of a near-symplectic form on a given broken Lefschetz pencil (for which the fibers are symplectic, called a \emph{near-symplectic pencil}), allows us to 
reproduce the following result of \cite{B1} without appealing to approximately holomorphic techniques.

\begin{cor} \label{simpBLP}
Any closed oriented connected $4$--manifold $X$ with $b_2^+(X)>0$ 
admits a near-symplectic simplified 
broken Lefschetz pencil.
\end{cor}

\begin{proof}
Since $b_2^+(X)>0$, there is a closed oriented surface 
$F$ in $X$ with $F^2>0$.  We can run the same procedure in the 
proof of Theorem~\ref{thmB} with this $F$ and \emph{any} (non-empty) null-homologous connected 
$1$--manifold $Z$ in $X$. 
(Once again, the twisting data is governed by the topology of $X$.) The result is a simplified broken Lefschetz pencil 
$f$ on $X$. Fibers are all homologous to $F$, which has positive square. Using the cohomology class which is the 
Poincar\'{e} dual of $[F]$, we can employ the Thurston--Gompf type construction of \cite{ADK} to build a 
near-symplectic form $\omega$ on $X$, with respect to which all fibers of $f$ are symplectic 
away from the singular locus. 
\end{proof}

\vspace{0.1in}
\section{Constructions of simplified trisections of $4$--manifolds} \label{Sec:Trisections}  

In this last section, we give algorithmic constructions of trisections of \linebreak \mbox{$4$--
manifolds,} which will utilize homotopy modifications of generic maps discussed in earlier sections. Our first goal is to  describe a correspondence between broken Lefschetz fibrations and trisections of $4$--manifolds. Meanwhile, we will prove that one can derive a rather special Gay--Kirby  trisection of any $4$--manifold from a given broken Lefschetz fibration on it, which we will call a \emph{simplified trisection}. We will then move on to presenting various new constructions of (simplified) trisections of $4$--manifolds using this dictionary.

\subsection{Broken Lefschetz fibrations to trisections and back} \label{Subsec:BLFandTS}  \

We begin with discussing how to derive a trisection from a broken Lefschetz fibration. Since we already have procedures to homotope any given broken Lefschetz fibration to a simplified (or a bit more generally, fiber-connected, directed)  broken Lefschetz fibration, we will content ourselves with presenting our arguments for such broken Lefschetz fibrations. As in the original proof of Gay and Kirby in \cite{GK4}, the trisection we obtain will be induced by a certain generic map (trisected Morse $2$--function) to the $2$--sphere. However, the trisections we get will be rather special: they do not have \emph{any} non-trivial ``Cerf boxes'' (where handle slides occur due to crossings between indefinite fold circles) and cusps only appear in triples (on the same singular circle). See Figure~\ref{fig511} below.  Following our earlier terminology, we call such a trisection a \emph{simplified trisection}. From such a generic map, one can obtain a trisection decomposition by looking at the three sectors shown in \mbox{Figure~\ref{fig511} (b),} in the same way as argued in \cite{GK4}.

\begin{thm}[Broken Lefschetz fibrations to trisections] \label{BLFtoTS}
Let $X$ admit a  fiber-connected, directed broken Lefschetz fibration $f\colon X \to S^2$ with embedded round image. Let $f$ have $k \geq 0$ Lefschetz singularities, $\ell \geq 0$ round locus components, and lowest regular fiber genus $g$. Then there exists a simplified $(g',k')$--trisection of $X$, with $(g',k')=(2g+k+\ell+2, 2g+\ell)$.
\end{thm}

\begin{proof}
For brevity, we start with discussing the case of simplified broken Lefschetz fibrations.
Let $f\colon X \to S^2$ be a genus $g+1$ fiber-connected simplified broken Lefschetz fibration with non-empty round locus. The \emph{genus of a simplified broken Lefschetz fibration} here is defined as the genus of the \emph{higher side} generic fiber. Consider the decomposition $S^2 \cong D^2_+ \cup (S^1 \times [-1, 1]) \cup D^2_-$, where $D^2_+$ (resp.\ $D^2_-$) is contained in the interior of the northern (resp.\ southern) hemisphere, $S^1 \times \{0\}$ corresponds to the equator,  $S^1 \times [-1, 1]$ is a regular neighborhood of the equator, and the three pieces are glued along their boundaries.  We may assume that all the Lefschetz critical values and the round image are contained in $\Int{D^2_+}$ in such a way that the round image is outward-directed in $D^2_+$.
Then, we have a trivialization $\rho\colon \, f^{-1}(S^1 \times [-1, 1]) \to S^1 \times [-1, 1] \times \Sigma_g$, where the restriction of $f$ is the composition of $\rho$ with the projection $S^1 \times [-1, 1] \times \Sigma_g \to S^1 \times [-1,1]$. We will show how to ``fold'' the fibration $f$ along this region to get a new generic map with a definite fold, and then simplify it to obtain the desired generic map yielding a trisection.

Let $h : \Sigma_g \to [1, 2]$ be a standard Morse function 
with exactly $2g$ index--$1$ critical points, one index--$0$ critical point, $x_1$, 
and one index--$2$ critical point, $x_2$, for which $h(x_1) = 1$ and $h(x_2) = 2$.
Define the smooth function 
$\varphi \colon [-1, 1] \times \Sigma_g
\to [1, 3]$ by $\varphi(t, x) = h(x) \cos (\pi t/2)+1$, 
$(t, x) \in  [-1, 1] \times \Sigma_g$
(see Figure~\ref{fig310}). 
We can easily check that $\varphi$ is a Morse function with
$\varphi^{-1}(1) = \{-1, 1\} \times \Sigma_g$
and that its critical points $\text{Crit}(\varphi)$ coincide 
with those of $h$ on $\{0\} \times \Sigma_g$: i.e.\ we have 
$\text{Crit}(\varphi)= \{0\} \times \text{Crit}(h)$. 
Furthermore, a critical point of index $\lambda$
for $h$ corresponds to a critical point of index $\lambda +1$ for $\varphi$.

\begin{figure}[htbp]
\centering
\psfrag{f}{$\varphi$}
\psfrag{0}{$0$}
\psfrag{1}{$1$}
\psfrag{2}{$2$}
\psfrag{3}{$3$}
\psfrag{si}{$\Sigma_g$}
\includegraphics[width=\linewidth,height=0.27\textheight,
keepaspectratio]{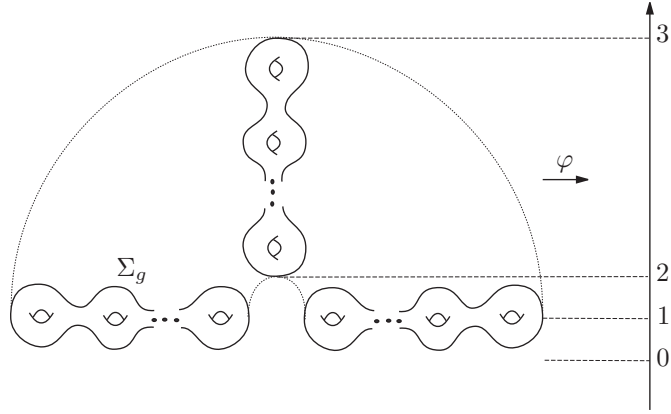}
\caption{Morse function $\varphi$.}
\label{fig310}
\end{figure} 

Let $\pi_\pm \colon D^2_\pm \to \R^2$ be the standard projections
$S^2 \to \R^2$ of the unit $2$--sphere restricted to $D^2_\pm$
composed with an appropriate multiplication by a positive constant
so that their image coincides with the unit disk in $\R^2$.
Define the smooth map \mbox{$g_0 \colon X \to \R^2$} by
$g_0|_{f^{-1}(D^2_\pm)} = \pi_\pm \circ f$
and $g_0|_{f^{-1}(S^1 \times [-1, 1])} = \eta \circ (\id_{S^1} \times \varphi) \circ \rho$,
where \mbox{$\eta \colon S^1 \times [1, 3] \hookrightarrow \R^2$}
is an appropriate embedding.

Then, $g_0$ has folds and Lefschetz singularities and its fold image consists of
concentric circles. The innermost one corresponds to the original round image $R_0$
and is outward-directed. The second one is inward-directed, and
the others are outward-directed except the outermost one which is a definite fold image.
Using $\mathrm{R2}^0$ and $\mathrm{R2}_2$ moves we can change the order of the 
first two circles, 
so now only the innermost circle is inward-directed. This, too can be reverted 
using flip and slip 
and push moves. However, unlike earlier, here we only
unsink \emph{one} of the $4$ cusp points on the reverted circle. 
Now the round image is directed 
outwards; the innermost circle has $3$ cusps, and all the Lefschetz critical values 
(including one new point) are in the central region.

\begin{figure}[htbp]
\centering
\psfrag{a}{(a)}
\psfrag{b}{(b)}
\includegraphics[width=\linewidth,height=0.3\textheight,
keepaspectratio]{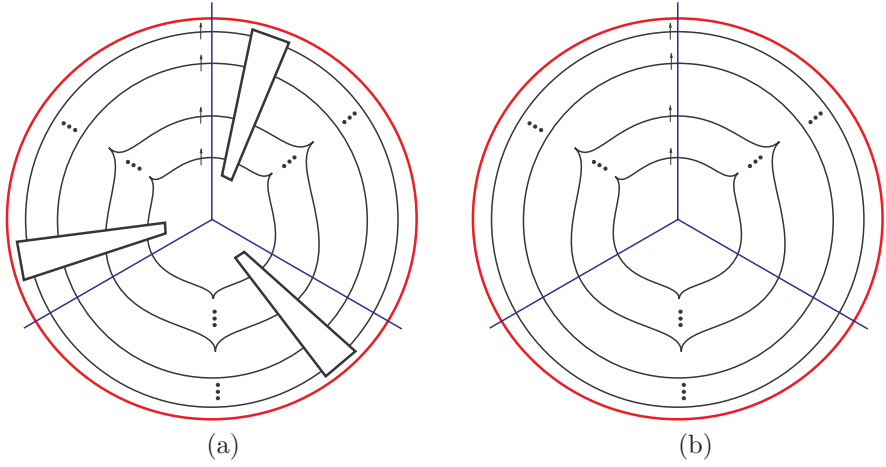}
\caption{(a) Singular image of general Gay--Kirby trisections with Cerf boxes; (b) simplified Gay--Kirby trisections we produce.}
\label{fig511}
\end{figure}

We can now wrinkle one of the Lefschetz singularities to get a ``triangle''; 
an indefinite fold circle with exactly $3$ cusps as in Figure~\ref{fig201}, 
whose image is embedded and directed outwards. Push all the other Lefschetz 
singularities into this triangle, and repeat the same procedure until no 
Lefschetz singularity is left. We finally get a generic map 
$g\colon X \to \R^2$, which has directed, embedded round image, 
with an embedded definite fold image as its outermost circle; 
see Figure~\ref{fig511} (b). 
The innermost $k+2$ circles all have $3$ cusps, where $k$ is the number of 
Lefschetz critical points of the original simplified broken Lefschetz fibration 
$f\colon X \to S^2$. The remaining $2g+1$ indefinite circles contain no cusps. 

The map $g\colon X \to \R^2$ prescribes a $(2g+k+3, 2g+1)$--trisection of 
$X$. (As in \cite{GK4}, this map yields a trisection decomposition of $X$ by looking at the three sectors shown in Figure~\ref{fig511} (b).)

More generally, assume that we have a fiber-connected, directed broken Lefschetz fibration $f \colon X \to S^2$ with embedded round image, where $Z_f$ has $\ell \geq 0$ components. Let the smallest genus of a regular fiber be $g$ and $k$ denote the number of Lefschetz singularities. We can then run the above construction to similarly derive a \mbox{$(2g+k+\ell+2, 2g+\ell)$--trisection.}
\end{proof}

\smallskip
\begin{rk}[Handle slides in simplified trisections]
The main difference between a general trisection and a simplified one is in the \emph{hierachy of handle slides}, which is imposed by the special structure of the simplified trisections. If we take any radial cut of the base disk from the center to its boundary, while avoiding the cusp points, the inverse image of this ray is a genus--$g$ handlebody, given by $g$ disjoint embedded simple closed curves $\{\alpha_1, \alpha_2, \ldots, \alpha_g\}$ on a central reference fiber $F \cong \Sigma_g$. Each $\alpha_i$ comes from the fiberwise $2$--handle attachment prescribed by the corresponding \mbox{$(g+1-i)$--th} indefinite fold arc image the ray crosses over, for $i=1, 2, \ldots, g$. \linebreak In a general trisection, moving across a \emph{non-trivial} Cerf box, these $\alpha_i$ can slide over each other in any fashion. In particular, the roles of any two $\alpha_i$ and $\alpha_j$ with $i \neq j$ might be interchanged. In a simplified trisection however, an $\alpha_i$ slides over $\alpha_j$ only if $i > j$. (That is, these handle slides only occur in an ``upper-triangular fashion'' in simplified trisections.)
\end{rk}

\begin{rk}[Trisections from --simplified broken-- Lefschetz fibrations] \label{SimpleCases}
For a genus--$(g+1)$ \emph{simplified} broken Lefschetz fibration $f\colon X \to S^2$ with non-empty round locus and $k$ Lefschetz singularities, we get a $(2g+k+3, 2g+1)$--trisection on $X$. For an honest genus--$g$ Lefschetz fibration, we get a $(2g+k+2, 2g)$--trisection. One can also allow $f$ to have \emph{achiral} Lefschetz singularities, where we also accept local orientation-preserving models $(z_1, z_2) \mapsto z_1 \bar{z}_2$ around singular points. The \emph{base diagram} of the trisection map itself will be insensitive to achirality. 

Our construction of trisections from Lefschetz fibrations can be seen to be complementary to Gay's work in \cite{Gay}, where he produces trisections from Lefschetz \emph{pencils}, which \emph{always have base points}. That is a very different construction, and the author points out that it does \emph{not} work for Lefschetz fibrations \cite[Remark~7]{Gay}. 
\end{rk}

\begin{rk}
It should be clear from our proof that we can also derive a Gay--Kirby trisection by ``folding'' any fiber-connected, directed broken Lefschetz fibration, which does not necessarily have an embedded image. Though in this case, the resulting trisection will not be simplified either.
\end{rk}

Together with Corollary~\ref{simpBLF} from the previous section, the construction in Theorem~\ref{BLFtoTS} establishes the existence of simplified trisections on arbitrary $4$--manifolds. 

\begin{cor} [Existence of simplified trisections] \label{STSexistence}
Any closed connected oriented $4$--manifold $X$ admits a simplified trisection, and such a trisection can be constructed algorithmically from any given generic map from $X$ to $S^2$.
\end{cor}

\smallskip
We also have the converse result:

\begin{prop}[Trisections to broken Lefschetz fibrations] \label{TStoBLF}
Let $X$ admit a \linebreak {$(g', k')$--trisection.} Then there exists a fiber-connected, directed broken  Lefschetz \linebreak fibration $f\colon X \to S^2$, which has regular fibers of highest genus $g$ and with $k$ Lefschetz singularities, where $g=g'+2$ and $k=3g'-3k'+4$. If  the given trisection is simplified, then $f$ in addition has embedded round image. 
\end{prop}

\begin{proof}
By embedding $\R^2$ into $S^2$ in such a way
that the central region contains the north pole, we consider the generic
map given by the trisection as a map into $S^2$. We then use a version of flip and slip for the definite fold, which first appeared in \cite{W} (where the author attributes the idea to Gay) and later in \cite[Fig~7]{GKPNAS}. It is depicted by a series of base diagram moves in Figure~\ref{figdefiniteflipslip}. By arguments identical to those we have for the indefinite case, the versions of flips and unflips involving definite folds here are always-realizable, and so are the Reidemeister type $\mathrm{R2}_2$ moves. Now  the image of the definite fold circle turns
into that of an indefinite fold circle which is directed
towards the image of the original map.

\begin{figure}[htbp]
\centering
\psfrag{I}{$\mathrm{R2}$ moves}
\psfrag{II}{$\mathrm{R2}^0$ and $\mathrm{R2}_1$}
\psfrag{B}{birth}
\psfrag{M}{cusp merge}
\psfrag{B1}{death}
\psfrag{M1}{fold merge}
\psfrag{0}{$0$}
\psfrag{1}{$1$}
\includegraphics[width=\linewidth,height=0.4\textheight,
keepaspectratio]{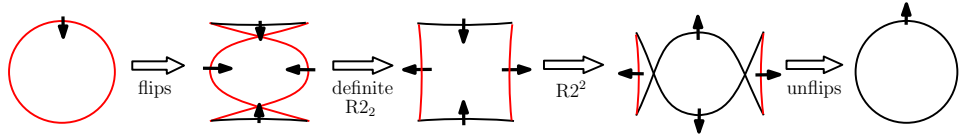}
\caption{Turning a definite fold circle to an indefinite one. Red lines depict definite fold images,
while black lines depict indefinite fold images. The arrows on definite folds indicate the fiberwise index--$3$ handle attachment direction.}
\label{figdefiniteflipslip}
\end{figure}

By always-realizable $\mathrm{R2}^0$, $\mathrm{R2}_2$, $\mathrm{R3}_3$ and $C$--moves, we can move the new indefinite circle to the north pole. We then apply flip and slip to it around the north pole. Unsink all the cusps and push them to the innermost region of this fiber-connected, directed broken Lefschetz fibration. The genus of a regular fiber in the innermost circle is now $g=g'+2$, and we get exactly $k=3(g'-k')+4$ Lefschetz critical points.
\end{proof}

\begin{rk}
The original proof of Gay--Kirby for the existence of trisections on arbitrary $4$--manifolds in \cite{GK4} also goes through --a rather different-- explicit sequence of homotopy modifications of generic maps. Our explicit procedures for simplifying broken Lefschetz fibrations allows our construction not to miss on this nice aspect. Given an arbitrary Gay--Kirby $(g',k')$--trisection, we can apply Proposition~\ref{TStoBLF} and Theorem~\ref{mainthm2} to eliminate the non-trivial Cerf boxes. Then applying Theorem~\ref{BLFtoTS}, we get back a simplified $(g'',k'')$--trisection ---typically with \emph{genus} $g''$ higher than the original genus $g'$.
\end{rk}

Here are a couple families of examples, where simplified trisections appear quite naturally:

\begin{example} [$\Sigma_g$--bundles over $S^2$]  When $g \geq 2$, these are of course isotopic to trivial fibrations on product manifolds $\Sigma_g \times S^2$. By Remark~\ref{SimpleCases}, we can derive simplified $(2g+2, 2g)$--trisections on them; 
cf.\ $(8g+5,4g+1)$--trisections of Gay--Kirby on the same examples 
\cite[pp.~3107--3108]{GK4}. 
In particular, we get $(2,0)$--trisections on $S^2 \times S^2$ and $\CP \# \CPb$.
\end{example}

\begin{example} [$S^1 \times Y^3$] \label{exheegaard}
Given a Heegaard splitting of a closed orientable 
$3$--manifold $Y$, we can almost simultaneously derive a simplified trisection 
and a fiber-connected, directed broken Lefschetz fibration on $S^1 \times Y$ with an embedded 
round image. This is based on a construction of \cite{BeH}. The genus--$g$ Heegaard splitting gives a certain family of Morse functions, any member of which can be slightly perturbed to obtain a Morse function $h: Y\to \R$ with unique minima and maxima mapping to $-g-1$ and~$g+1$, respectively, where the Heegaard surface is at the level $0$, and  the index $1$ critical points map to integer values in $(-g-1,0)$ and the index $2$ critical points map to integer values in $(0, g+1)$. Composing the generic map 
\[\text{id}_{S^1} \x h \colon S^1 \x Y \to S^1 \x [-g-2,g+2] \]
with an embedding of the annulus  \[S^1\x [-g-2, g+2] \to S^2 ,\] 
we obtain a generic map $f_0 \colon S^1 \times Y \to S^2$, with embedded singular image as given in Figure~\ref{figheegaard}. All are indefinite folds except for the outermost two definite folds. 

We can derive a simplified trisection by homotoping $f_0$ as follows: First trade the definite fold circle around the south pole with an indefinite fold circle following the base diagram move in Figure~\ref{figdefiniteflipslip}. Then apply a pair of $\mathrm{R2}^0$ and $\mathrm{R2}_2$ moves $g$ times to bring this new indefinite fold circle to the equator. Now we can apply a version of the flip and slip to the southernmost indefinite fold, where we unsink only one of the cusps. Now applying three $C$--moves and three $\mathrm{R2}_2$ moves a total of $g-1$ times, we can bring this ``triangle'' next to the equator. Repeating the same step $g-1$ more times for each indefinite fold directed towards the south pole (while pushing the new Lefschetz critical points across them), we arrive at a nested family of indefinite folds all directed towards the north pole, where $g+1$ of the innermost circles do not have any cusps, and the next $g$ circles have triples of cusps. As in the proof of Theorem~\ref{BLFtoTS}, we can use a sequence of wrinkling and push moves so the $g$ Lefschetz singularities around the southern pole give additional $g$ nested circles with triples of cusps. The resulting Morse $2$--function $f_1 \colon S^1 \x Y \to S^2$ prescribes a simplified $(3g+1, g+1)$--trisection. (cf. the $(6g+1, 2g+1)$--trisections on $S^1 \x Y$ constructed in 
\cite[pp.~3104--3107]{GK4} for a given genus--$g$ Heegaard splitting of $Y$.) 

Note that this construction gives the minimal genus trisection on $S^1 \x S^3$, a (simplified) $(1,1)$--trisection. Lastly note that for $T^2 \times S^2 = S^1 \x (S^1 \x S^2)$ the two viewpoints taken in this remark and the previous one both yield a (simplified) $(4,2)$--trisection.

\begin{figure}[!h]
\centering
\includegraphics[width=\linewidth,height=0.3\textheight,
keepaspectratio]{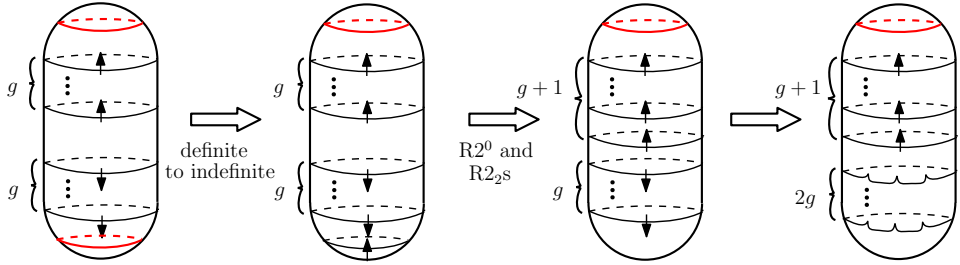}
\caption{Deriving a simplified $(3g+1,g+1)$--trisection on $S^1 \x Y$ from a genus--$g$ Heegaard splitting of $Y$. Definite folds are shown in red. In the very last step, one by one, we turn each one of the $g$ indefinite fold circles in the southern hemisphere inside-out by a version of flip and slip, where we unsink only one of the four cusps, and then place it next to the equator using $C$--moves and $\mathrm{R2_2}$s. Unsinked $g$ Lefschetz singularities are pushed across these folds, and at the end are wrinkled to add additional $g$ circles with triples of cusps.}
\label{figheegaard}
\end{figure} 

By homotoping $f_0$ a bit differently, we can also derive a fiber-connected, directed broken Lefschetz fibration on $S^1 \times Y$ with embedded round image. We go through almost the same homotopy from $f_0$ to $f_1$, except we unsink all cusps and push them to a small neighborhood of the south pole, and do not wrinkle any Lefschetz singularities. Thus we have an embedded singular image, where the base diagram consists of a definite fold enclosing outward-directed $2g+1$ indefinite fold circles, and $4g$ Lefschetz singularities around the south pole. Now, as in the proof of Proposition~\ref{TStoBLF}, we trade the remaining definite fold around the north pole with an indefinite one, apply several $\mathrm{R2}^0$, $\mathrm{R2}_2$, and push moves to bring it around the south pole, and apply flip and slip. The resulting map $f_2\colon S^1 \times Y \to S^2$ is a desired broken Lefschetz fibration, whose highest genus regular fiber has genus $2g+3$ and it has $4g$ Lefschetz singularities.
\end{example}

\begin{rk}
Although any two trisection decompositions on a $4$--manifold $X$ are stably equivalent \cite{GK4}, there are no examples known to require an arbitrarily large number of stabilizations to become equivalent. Our construction in Example~\ref{exheegaard} makes it plausible that 
such examples can be catered from similar examples for Heegaard splittings. We hope to turn to this problem in future work. For instance, by the work of Hass, Thompson and Thurston in \cite{HTT}, for each $g \geq 2$, there is a $3$--manifold $Y_g$ with two genus--$g$ Heegaard splittings that require at least $g$ stabilizations to become equivalent. 
\textit{For such pairs of Heegaard splittings, how many stabilizations are needed for the corresponding simplified trisections on $S^1 \x Y_g$ to become equivalent?} 
\end{rk}

\smallskip
\subsection{Exotic trisections} \label{Subsec:exotic}  \

The list of $4$--manifolds that admit $g'=0$, $1$, or $2$ trisections is very short: the only $g'=0$ trisection is for $S^4$ and $g'=1$ examples are for $\CP$, $\CPb$ and $S^1 \x S^3$ \cite{GK4}, whereas a $g'=2$ trisection is either a connected sum of these $g'=1$ examples or is a trisection of $S^2 \x S^2$ \cite{MZ}. 
Thus Meier and Zupan raise the following natural question \cite[Question~1.2]{MZ}:\ \textit{What is the smallest value of $g'$ for which there are infinitely many $4$--manifolds which admit $(g',k')$--trisection for some $k'$?}

Theorem~\ref{BLFtoTS}, combined with earlier constructions of simplified broken Lefschetz fibrations, answers this question:

\begin{cor} \label{MZquestion}
For every fixed $g' \geq 3$ and $g' -2 \geq k' \geq 1$, there are infinitely many homotopy inequivalent $4$--manifolds admitting $(g',k')$--trisections.
\end{cor}

\begin{figure}[!h]
\centering
\includegraphics[width=\linewidth,height=0.3\textheight,
keepaspectratio]{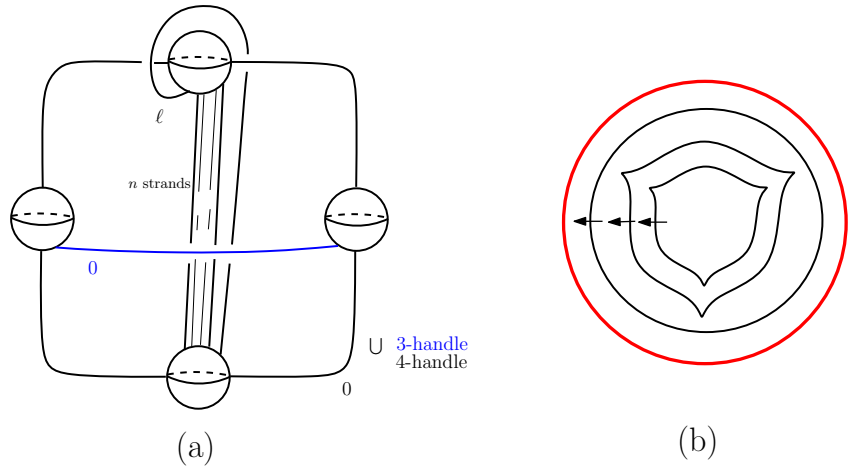}
\caption{(a) Kirby diagram for genus--$1$ simplified broken Lefschetz fibrations on rational homology $4$--spheres $L_n$ and $L'_n$, for $\ell$ even and odd, respectively. The $2$--handle and the $3$--handle that make up the round $2$--handle are given in blue. For $n=1$ we have $S^4$. (b) The base diagram for a simplified trisection on $L_n$ and $L'_n$ we obtain from these simplified broken Lefschetz fibrations.}
\label{figpao}
\end{figure} 

\begin{proof}
As shown in \cite{BK, H1}, an infinite family of homotopy inequivalent $4$--manifolds, namely the rational homology $4$--spheres $L_n$ and $L'_n$ (say for $n \geq 2$), admit genus--$1$ simplified broken Lefschetz fibrations with no Lefschetz singularities. A handlebody description of these simplified broken Lefschetz fibrations (following \cite{B1}) is given in Figure~\ref{figpao} (a). It is easy to see that $H_1(L_n; \Z)=H_1(L'_n; \Z)=\Z_n$ for $n \geq 2$.

By Theorem~\ref{BLFtoTS}, these yield simplified $(3,1)$--trisections; see Figure~\ref{figpao} (b) for the base diagram. We can use connected sums with standard trisections \cite{GK4} (e.g.\ with standard $(1,0)$ and $(1,1)$ trisections on $\CPb$ and $S^1 \times S^3$, respectively) to arrive at the other higher $(g',k')$ examples in the statement. 
\end{proof}

\begin{rk}
In \cite{MZ}, Meier and Zupan announce that their work in progress with Gay will provide an infinite family of $(3,1)$--trisections, which they construct as double branched covers of $S^4$ along $k$--twist spun $2$--bridge knots. It would be interesting to compare these examples. Akin to the classification of low genera simplified broken Lefschetz fibrations \cite{H1, BK}, one might expect the list of (simplified) $(3,1)$--trisections to consist of fairly standard ones. 
\end{rk}

\begin{example} 
An interesting example is the simplified $(3,1)$--trisection we derive from the genus--$1$ simplified broken Lefschetz fibration on $S^4$ \cite{ADK}. The Morse $2$--function for this trisection we get has 2 indefinite circles with 3 cusps on each, and one outer indefinite circle with no cusps; see Figure~\ref{figpao} for its base diagram. (The \mbox{$(3,1)$--trisections} of infinitely many rational homology spheres we obtained in the proof of Corollary~\ref{MZquestion} all have the same base diagram.) On the other hand, the ``standard trisection'' of $S^4$ \cite[Figure~27]{GK4} used for stabilization comes from a Morse $2$--function with 3 indefinite circles with 2 cusps on each. After three $C$--moves, one can obtain the same base diagrams outside of the Cerf boxes. 
\end{example}

\smallskip

A natural idea for relating newly emerging theory of trisections of $4$--manifolds to constructions of exotic smooth structures is to find $(g',k')$--trisections, whose three sectors can be re-glued in a way the diffeomorphism type is changed, but the homeomorphism type remains the same. Perhaps what is more interesting than the above question is to find the smallest $g'$ for which there are infinitely many non-diffeomorphic $4$--manifolds \emph{in the same homeomorphism class} admitting a $(g',k')$--trisection, for some $k'$ ---what one can call (small) \emph{exotic trisections}. We finish with presenting some examples:

\begin{cor} \label{exotic}
There is an infinite family of exotic $(34,8)$--trisections in the homeomorphism class of $\CP\#9 \CPb$. There is an exotic $(20,4)$--trisection in the homeomorphism class of $\CP \# 7 \CPb$. 
\end{cor}

\begin{proof} 
As shown by Fintushel and Stern, a knot surgered elliptic surface $E(1)_K$ admits a genus--$2h$ Lefschetz fibration, for $K$ a genus--$h$ fibered knot \cite{FS}. For a family of genus--$2$ fibered knots $K_i$ with distinct Alexander polynomials, we obtain a family of genus--$4$ Lefschetz fibrations $(X_i, f_i)$, where $X_i$ are mutually non-diffeomorphic irreducible symplectic \mbox{$4$--manifolds} all in the same homeomorphism class.
Per Remark~\ref{SimpleCases}, we derive an infinite family of exotic $(34, 8)$--trisections from these examples. Standard \mbox{$E(1)= \CP \# 9 \CPb$} also admits a $(34,8)$--trisection: Applying our procedure to the elliptic Lefschetz fibration on $E(1)$ (where $g=1$, $k=12$), we obtain a \mbox{$(16,2)$--trisection,} which we can then stabilize $6$ times (with a $(3,1)$--trisection on $S^4$) to get a $(34,8)$--trisection on $E(1)$ as well.

For a smaller (but a single) example, we can take the genus--$2$ Lefschetz fibration on an exotic, irreducible symplectic $\CP \# 7 \CPb$ constructed in \cite{BaykurKorkmaz}, and apply Theorem~\ref{BLFtoTS} to produce a $(20,4)$--trisection. The standard $\CP \# 7 \CPb$ also admits a $(20,4)$--trisection: Take a rational Lefschetz fibration on it with $6$ singular points, apply our procedure, and then stabilize it $4$ times. 
\end{proof}

\smallskip
\begin{rk}
Per the very nature of our article, we have only discussed trisections as a certain class of generic maps in this article. It would be interesting to analyze our examples, especially the exotic pairs, in terms of induced trisection diagrams, which can be derived following the explicit procedures given in \cite{CGP, Castro}. 
\end{rk}


\end{document}